\documentclass[11pt]{amsart}

\usepackage{stmaryrd}

\usepackage{cite}

\usepackage{tikz}
\usetikzlibrary{decorations.pathreplacing}
\usepackage{float}

\usepackage{amsmath,amssymb}
\usetikzlibrary{patterns}

\usepackage{amsfonts,color}
\usepackage{latexsym,amssymb}
\usepackage{amsmath}
\usepackage[utf8]{inputenc}
\usepackage{enumitem}
\usepackage{comment}

\usepackage[
pagebackref,
colorlinks=true,
urlcolor=blue,
citecolor=red,
linkcolor=blue,
linktocpage,
pdfpagelabels,
bookmarksnumbered,
bookmarksopen
]{hyperref}

\usepackage{ragged2e}
\usepackage{tikz}
\usepackage{float}
\usepackage{booktabs}
\usepackage{subfig}

\newtheorem{theorem}{Theorem}[section]
\newtheorem{lemma}[theorem]{Lemma}
\newtheorem{proposition}[theorem]{Proposition}
\newtheorem{corollary}[theorem]{Corollary}
\newtheorem{remark}{Remark}[section]
\newtheorem{definition}[theorem]{Definition}

\catcode`@=11 \@addtoreset{equation}{section} \catcode`@=12

\usepackage[left=3cm,right=3cm,top=3cm,bottom=3cm]{geometry}

\newcommand{\dx}{\,\mathrm{d}x}

\begin{document}
		
\title[Nodal solutions for a zero-mass \((p,q)\)-Laplacian equation]{
Nodal radial solutions for a zero-mass \((p,q)\)-Laplacian equation with Hardy-Hénon type weight
}

\author[J. C. de Albuquerque]{J. C. de Albuquerque}
\address{Departamento de Matem\'atica, Universidade Federal de Pernambuco \\
	50670-901, Recife--PE, Brazil}
\email{josecarlos.melojunior@ufpe.br}
\thanks{The first author was partially supported by CNPq with grant 301849/2025-0.}

\author[J. Carvalho]{J. Carvalho}
\address{Universidade Federal de Sergipe, Departamento de Matem\'atica, 49100-000 São Cristóvão-SE, Brazil}
\email{jonison@mat.ufs.br}\thanks{The second author was partially supported by CNPq with grant 300853/2025-4.} 

\author[E. Barboza]{E. Barboza}
\address{Departamento de Matem\'atica, Universidade Federal Rural de Pernambuco \\
	50740-560, Recife--PE, Brazil}
\email{eudes.barboza@ufrpe.br}
\thanks{The third author was partially supported by FACEPE/Brazil Proc. APQ-0967-1.01/24.}

		
		
\begin{abstract}

We establish the existence of radial nodal solutions for a class of zero-mass \((p,q)\)-Laplacian equations in \(\mathbb{R}^N\) with a Hardy–Hénon type weight \(|x|^\alpha\). The absence of lower-order potential terms motivates the use of a homogeneous radial variational framework adapted to the mixed \(p\)- and \(q\)-growth structure. A key ingredient is a weighted Sobolev embedding theorem that establishes continuity and compactness in the appropriate range of exponents and identifies the sharp interval for continuous embedding. Building on these results, we construct sign-changing radial solutions through a variational method based on radial partitions and critical levels. The main result asserts that, for every \(k\in\mathbb{N}\), the problem admits a radial weak solution with exactly \(k\) nodal zeros. The proof combines minimization on Nehari manifolds over radial subdomains with an optimization procedure that determines the nodal radii and ensures the compatibility of the solution across nodal interfaces.
\end{abstract}
		
\maketitle



\section{Introduction}

In this work we consider the following class of quasilinear elliptic equations driven by $(p,q)$-Laplacian operator
\begin{equation}\label{eq1}\tag{$\mathcal{P}$}
-\Delta_p u-\Delta_q u
=
|x|^\alpha |u|^{s-2}u,
\quad
\text{in } \mathbb{R}^N,
\end{equation}
where
$
1<p<q<N$, $N\geq 3$, $\alpha> -p,
$  
and \(\Delta_m u:=\operatorname{div}(|\nabla u|^{m-2}\nabla u)\), \(m>1\), denotes the \(m\)-Laplacian.  The nonlinearity in \eqref{eq1} contains the weight \(|x|^\alpha\), which may be singular at the origin if \(\alpha<0\) which is called Hardy-type weight, corresponds to the unweighted case if \(\alpha=0\), and becomes a Hénon-type weight if \(\alpha>0\). For this reason, \(|x|^\alpha\) is usually regarded as a Hardy--Hénon type weight.

In the last decades, quasilinear equations involving the \((p,q)\)-Laplacian operator have been extensively studied. A general model is given by
\begin{equation}\label{problema-inicial}
-\Delta_p u-\Delta_q u
+
V(x)|u|^{p-2}u
+
K(x)|u|^{q-2}u
=
f(x,u),
\quad
\text{in } \mathbb{R}^N,
\end{equation}
where \(V,K:\mathbb{R}^N\to\mathbb{R}\) are potentials and
\(f:\mathbb{R}^N\times\mathbb{R}\to\mathbb{R}\) is a nonlinear term. As mentioned in \cite{Cher-yasov}, this class of problems arises from reaction-diffusion models of the form
\[
u_t
=
\operatorname{div}\left(\mathcal{D}_{p,q}(\nabla u)\nabla u\right)
+
\mathcal{C}(x,u),
\]
where
$
\mathcal{D}_{p,q}(\nabla u)
=
|\nabla u|^{p-2}+|\nabla u|^{q-2}
$
is a nonlinear diffusion coefficient and
\[
\mathcal{C}(x,u)
=
-V(x)|u|^{p-2}u
-
K(x)|u|^{q-2}u
+
f(x,u)
\]
is the reaction term. From this perspective, \(u\) may represent a concentration, while the diffusion coefficient depends nonlinearly on the gradient. Problems of type \eqref{problema-inicial} also appear in models related to nonlinear diffusion, plasma physics, biophysics, chemical reaction design and electromagnetism, see for instance, \cite{Bonheure-Avenia-Pomponio,Bonheure-Iacopetti}.

When the potentials \(V\) and \(K\) are bounded from below by a positive constant, variational methods can usually be developed in nonhomogeneous Sobolev-type spaces, since the terms involving \(V\) and \(K\) provide control of the \(L^p\)- and \(L^q\)-norms of \(u\). In this direction, we refer the reader to
\cite{Ambrosio,Alves-Olimpio-Assuncao,He-Li,Figueiredo,Xie-Wang-Zhang,Wang-Wang-Zhang}
and the references therein. The situation is substantially different in the limiting case
$
V\equiv K\equiv0.
$
In this case, \eqref{problema-inicial} reduces to
\begin{equation}\label{problema-massa-zero}
-\Delta_p u-\Delta_q u
=
f(x,u),
\quad
\text{in } \mathbb{R}^N,
\end{equation}
which is known in the literature as the \textit{``zero-mass case''}. In this situation, the associated energy functional has only gradient terms in its principal part. Thus the usual Sobolev space involving the full norm of \(u\) is no longer the most appropriate framework. Instead, one is naturally led to homogeneous spaces, where the norm is defined only through the gradient. In this direction, see for example,
\cite{Liu-Perera,Baldelli-Brizi-Filippucci-JGA,He-Li,Pomponio-Watanabe} and references therein.

For \(1<\ell<N\), we denote by
\[
\ell_\alpha^*
:=
\frac{\ell(N+\alpha)}{N-\ell}
\]
the weighted critical exponent associated with the Sobolev-type embedding involving the weight \(|x|^\alpha\). In particular, the exponents
\[
p_\alpha^*
=
\frac{p(N+\alpha)}{N-p}
\qquad\text{and}\qquad
q_\alpha^*
=
\frac{q(N+\alpha)}{N-q}
\]
will play an important role in the analysis. 
Since \(1<p<q\), the assumption \(\alpha>-p\) implies, in particular, that
$
p_\alpha^*>p
$
and
$q_\alpha^*>q$. Moreover, under this assumption one has
$
p_\alpha^*<q_\alpha^*.
$
In order to treat the pure-power nonlinearity in \eqref{eq1} variationally and to obtain the usual geometry of the fibering maps, we shall assume that
\[
\max\{q,p_\alpha^*\}<s<q_\alpha^*.
\]
The condition \(s>q\) guarantees the superlinear character of the energy along one-dimensional rays, whereas the condition \(s<q_\alpha^*\) places the nonlinearity in the compact range of the weighted embedding. The following weighted Sobolev embedding provides the compactness required in our variational approach and, moreover, identifies the sharp range of exponents for the corresponding continuous embedding.

\begin{theorem}\label{embeddingRN}
    Assume that $1<p<q<N$ and $\alpha>-p$. Then the embedding \( E_{\mathrm{rad}}(\mathbb{R}^{N}) \hookrightarrow L^s(\mathbb{R}^N, |x|^\alpha)\) is continuous, for any $p^*_\alpha \leq s\leq q^*_\alpha$. Furthermore, this embedding is compact whenever $p^*_\alpha < s < q^*_\alpha$.
\end{theorem}
Here, \(E_{\mathrm{rad}}(\mathbb R^N)\) denotes the homogeneous radial space introduced in Section 2.
As a consequence of Theorem \ref{embeddingRN}, we obtain a weighted Sobolev embedding when \(\Omega\) is either a ball, an annulus, or an exterior domain centered at the origin, see subsection \ref{sub-embedding} for more details.

\begin{remark}[Sharpness of the embedding interval] The range 
$
p_\alpha^*\leq r\leq q_\alpha^*
$
of the continuous embedding in Theorem \ref{embeddingRN} is sharp. Indeed, take
$
u\in C_{0,\mathrm{rad}}^\infty(\mathbb{R}^N)\setminus\{0\}
$
and, for each \(\mu>0\), define
$
u_\mu(x):=u\left({x}/{\mu}\right)$. If the continuous embedding
$
E_{\mathrm{rad}}(\mathbb{R}^N)
\hookrightarrow
L^s(\mathbb{R}^N,|x|^\alpha)
$
holds, then there exists \(C>0\), independent of \(\mu\), such that
\begin{equation}\label{rem1}
\left(
\int_{\mathbb{R}^N}
|x|^\alpha |u_\mu(x)|^s\,\mathrm{d}x
\right)^{1/s}
\leq
C
\left(
\|\nabla u_\mu\|_{L^p(\mathbb{R}^N)}^2
+
\|\nabla u_\mu\|_{L^q(\mathbb{R}^N)}^2
\right)^{1/2}.
\end{equation}
By the change of variables \(y=x/\mu\), we have the following identities:
\begin{equation}\label{rem2}
\int_{\mathbb{R}^N}
|\nabla u_\mu(x)|^p\,\mathrm{d}x
=
\mu^{N-p}
\int_{\mathbb{R}^N}
|\nabla u(y)|^p\,\mathrm{d}y,
\end{equation}
\begin{equation}\label{rem3}
\int_{\mathbb{R}^N}
|\nabla u_\mu(x)|^q\,\mathrm{d}x
=
\mu^{N-q}
\int_{\mathbb{R}^N}
|\nabla u(y)|^q\,\mathrm{d}y,
\end{equation}
\begin{equation}\label{rem4}
\int_{\mathbb{R}^N}
|x|^\alpha |u_\mu(x)|^s\,\mathrm{d}x
=
\mu^{N+\alpha}
\int_{\mathbb{R}^N}
|y|^\alpha |u(y)|^s\,\mathrm{d}y.
\end{equation}
By combining \eqref{rem1}-\eqref{rem4} and using the elementary inequality \((a+b)^{1/2}\leq a^{1/2}+b^{1/2}\), we deduce
\[
\left(
\int_{\mathbb{R}^N}
|y|^\alpha |u(y)|^s\,\mathrm{d}y
\right)^{1/s}
\leq
C
\left(
\mu^{\gamma_p}
\|\nabla u\|_{L^p(\mathbb{R}^N)}
+
\mu^{\gamma_q}
\|\nabla u\|_{L^q(\mathbb{R}^N)}
\right),
\]
where
\[
\gamma_p:=
\frac{N-p}{p}-\frac{N+\alpha}{s} \quad \mbox{and} \quad
\gamma_q:=
\frac{N-q}{q}-\frac{N+\alpha}{s}.
\]
Since \(u\neq0\), the left-hand side is positive and independent of \(\mu\). Hence the right-hand side above cannot tend to \(0\). If
$
s<p_\alpha^*,
$
then
$
\gamma_p<0.
$
Since \(\gamma_q<\gamma_p\), we also have
$
\gamma_q<0.
$
Letting \(\mu\to+\infty\), we obtain
$
\mu^{\gamma_p}\to0
$ and $\mu^{\gamma_q}\to0$,
which leads to a contradiction. On the other hand, if
$
s>q_\alpha^*$,
then
$
\gamma_q>0.
$
Since \(\gamma_p>\gamma_q\), we also have
$
\gamma_p>0.
$
Letting \(\mu\to0^+\), we obtain
$
\mu^{\gamma_p}\to0
$ and $\mu^{\gamma_q}\to0,
$
which is not possible. Therefore, the continuity of the embedding can hold only if
$
p_\alpha^*\leq s\leq q_\alpha^*,
$
and the interval of exponents in Theorem \ref{embeddingRN} is sharp.
\end{remark}

\begin{definition}
We say that \(u\in E(\mathbb{R}^N)\) is a weak solution of \eqref{eq1} if
\[
\int_{\mathbb{R}^N}
|\nabla u|^{p-2}\nabla u\nabla \varphi\,\mathrm{d}x
+
\int_{\mathbb{R}^N}
|\nabla u|^{q-2}\nabla u\nabla \varphi\,\mathrm{d}x
=
\int_{\mathbb{R}^N}
|x|^\alpha |u|^{s-2}u\varphi\,\mathrm{d}x
\]
for every \(\varphi\in E(\mathbb{R}^N)\). A nontrivial weak solution \(u\) of \eqref{eq1} is called a \emph{nodal solution}
if it changes sign in \(\mathbb{R}^N\), that is,
$
u^+\not\equiv0
$
and
$
u^-\not\equiv0,
$
where
$
u^+:=\max\{u,0\},
$
$
u^-:=\max\{-u,0\}.
$
The connected components of the set
$
\{x\in\mathbb{R}^N:\ u(x)\neq0\}
$
are called the \emph{nodal domains} of \(u\). If \(u=u(r)\) is radial, a radius \(r_0>0\) is called a \emph{node} of \(u\)
if
$
u(r_0)=0
$
and \(u\) changes sign across \(r_0\). We say that \(u\) has exactly \(k\)
nodes if there exist
$
0<r_1<\cdots<r_k<\infty
$
such that \(r_1,\ldots,r_k\) are the only radii at which \(u\) changes sign.
In this case, \(u\) has exactly \(k+1\) nodal domains.
\end{definition}

The main result of the paper can be stated as follows.

\begin{theorem}\label{existencianodal}
Assume that $N\geq3$, $1<p<q<N$, $\alpha>-p$, and $\max\{q,p_\alpha^*\}<s<q_\alpha^*$. Then, for every \(k\in\mathbb{N}\), problem \eqref{eq1} admits a radial weak solution with exactly \(k\) nodal zeros. Equivalently, \eqref{eq1} has a radial sign-changing weak solution with exactly \(k+1\) nodal domains.
\end{theorem}

The strategy follows the method used by Bartsch and Willem \cite{Bartsch-Willem} (see also \cite{Cao-Zhu,Fan-Zhao, Nehari}) for the construction of nodal radial solutions. First, on each radial subdomain we obtain one-sign solutions by minimizing the energy on a suitable Nehari manifold. Then we associate to each radial partition a sum of least energy levels, one for each nodal region, and minimize this sum over the class of admissible partitions. The minimizer determines the nodal radii, while the corresponding one-sign solutions are glued with alternating signs. A delicate point in this construction is to prove that the resulting function is not only continuous, but also satisfies the correct matching condition at the nodal radii. In other words, one has to show that no jump of the radial derivative occurs at the interfaces between two consecutive nodal regions. This is obtained through a local perturbation argument around each nodal radius, combined with a sharp energy comparison. Once this matching property is established, the glued function becomes a radial weak solution of the original equation.

\begin{remark}
Although the main purpose of this paper is the construction of nodal radial solutions, the whole-space embedding result (Theorem \ref{embeddingRN}) also yields solutions of \eqref{eq1} with definite sign. Indeed, under the same assumptions on the exponent \(s\), the functional
\[
I_{\mathbb{R}^N}(u)
=
\frac{1}{p}\int_{\mathbb{R}^N}|\nabla u|^p\,\mathrm{d}x
+
\frac{1}{q}\int_{\mathbb{R}^N}|\nabla u|^q\,\mathrm{d}x
-
\frac{1}{s}\int_{\mathbb{R}^N}|x|^\alpha |u|^s\,\mathrm{d}x
\]
is well defined and of class \(C^1\) on \(E_{\mathrm{rad}}(\mathbb{R}^N)\). Moreover, the compact weighted embedding in the radial setting allows one to apply the usual minimization argument on the Nehari manifold to obtain a nontrivial radial critical point of \(I_{\mathbb{R}^N}\). Then we use a version of the principle of symmetric criticality (Proposition \ref{palais}) to conclude that this critical point is a weak solution of \eqref{eq1}. Since
$
I_{\mathbb{R}^N}(|u|)=I_{\mathbb{R}^N}(u),
$
the minimizer may be chosen nonnegative. Therefore, by the strong maximum principle, it is in fact positive in \(\mathbb{R}^N\). Finally, because the equation is odd with respect to \(u\), the function \(-u\) is also a weak solution. Therefore, the same variational framework provides both a positive and a negative radial weak solution of \eqref{eq1}.
\end{remark}

\begin{remark}
A similar approach can be applied to problem \eqref{eq1} with a general nonlinear term \(f(x,u)\), namely
\[
-\Delta_p u-\Delta_q u
=
|x|^\alpha f(x,u),
\quad
\text{in } \mathbb{R}^N,
\]
provided that \(f\) satisfies the usual Bartsch--Willem type assumptions, see \cite{Bartsch-Willem}. These hypotheses guarantee the variational structure, the compactness needed in the radial setting, the Nehari-type geometry on each radial subdomain, and the oddness required to construct positive and negative pieces. Therefore, the radial partition method used here also yields nodal radial solutions for such more general nonlinearities.
\end{remark}

\vspace{0,2cm}

\noindent \textbf{Outline.}
The remainder of the paper is organized as follows. In Section \ref{sec2}, we introduce the variational framework associated with the problem. We define the radial spaces used throughout the paper, prove the weighted Sobolev embeddings needed for the analysis, and present the weak formulation of the equation. Section \ref{sec3} is devoted to the construction of nodal radial solutions. More precisely, we first study positive and negative solutions on radial subdomains. Then we analyze qualitative properties of the associated critical values and use them to construct nodal solutions by minimizing suitable sums of critical levels over radial partitions. Finally, we prove Theorem \ref{existencianodal} by showing that the constructed nodal profile satisfies the required matching conditions at the nodal radii and hence gives a genuine weak solution of the original problem.

\section{The variational framework}\label{sec2}

This section is devoted to introducing the functional setting to the variational treatment of problem \eqref{eq1}. 

\subsection{The function spaces}

Since the equation is posed in the zero-mass case, no lower-order potential term is available to control the \(L^p\)- or \(L^q\)-norm of the solution. For this reason, we consider the space
\[
E(\mathbb{R}^{N})
:=
\overline{C_0^\infty(\mathbb{R}^N)}^{\|\cdot\|_{p,q}},
\]
with respect to the norm
\[
\|u\|_{p,q}
:=
\left(
\|\nabla u\|_{L^p(\mathbb{R}^N)}^2
+
\|\nabla u\|_{L^q(\mathbb{R}^N)}^2
\right)^{1/2}.
\]
We also introduce the radial homogeneous space
\[
E_{\mathrm{rad}}(\mathbb{R}^{N})
:=
\overline{C_{0,\mathrm{rad}}^\infty(\mathbb{R}^N)}^{\|\cdot\|_{p,q}}.
\]
Throughout this work, we also use the corresponding radial spaces on radial domains. More precisely, if
\(\Omega\subset\mathbb{R}^N\) is a ball, an annulus or an exterior domain centered at the origin, we define
\[
E_{\mathrm{rad}}(\Omega)
:=
\overline{C_{0,\mathrm{rad}}^\infty(\Omega)}^{\|\cdot\|_{p,q}},
\]
where
\[
\|u\|_{p,q,\Omega}
:=
\left(
\|\nabla u\|_{L^p(\Omega)}^2
+
\|\nabla u\|_{L^q(\Omega)}^2
\right)^{1/2}.
\]
When \(\Omega\) is an exterior domain of the form
\[
\Omega(R,\infty)
:=
\{x\in\mathbb{R}^N:\ |x|>R\}=\mathbb{R}^N\setminus B_R,
\]
the same definition gives the appropriate homogeneous finite-energy space. In this case, the elements of \(E_{\mathrm{rad}}(\Omega(R,\infty))\) have zero trace on the inner boundary \(\partial B_R\), while no boundary condition is imposed at infinity.

Henceforth, for $\Omega\subset\mathbb{R}^{N}$ and each $1\leq s<\infty$, the weighted Lebesgue space is defined by
$$
L^s(\Omega,|x|^{\alpha}):=\left\lbrace u:\Omega \to \mathbb{R} \ \mbox{measurable}  : \int_{\Omega}|x|^{\alpha}|u|^s\dx<\infty \right\rbrace, 
$$
equipped with the norm
$$
\|u\|_{L^s(\Omega,|x|^{\alpha})}:=\left(\int_{\Omega}|x|^{\alpha}|u|^s\dx\right)^{1/s}.
$$

\subsection{Weighted Sobolev embeddings for radial functions}\label{sub-embedding}

Inspired by \cite{Figueiredo-Moreira-Miyagaki,Su-Wang-Willem1,Su-Wang-Willem2}, we establish continuous and compact weighted Sobolev embeddings. Before proving the embedding for \(E_{\mathrm{rad}}(\mathbb{R}^N)\), we establish an auxiliary endpoint estimate in Sobolev space \(D_{\mathrm{rad}}^{1,m}(\mathbb{R}^N)\), with \(m\in\{p,q\}\). For this purpose, we define
\[
m^*:=\frac{mN}{N-m} \quad \mbox{and}
\quad
m_\alpha^*:=\frac{m(N+\alpha)}{N-m}.
\]
Since \(\alpha>-p\) and \(p<q<N\), we have in particular that \(m_\alpha^*>0\) for both \(m=p\) and \(m=q\).

The following result is a version of the radial
Caffarelli-Kohn-Nirenberg inequality established in \cite{Arka-Hoai},
for which we provide a more elementary proof.

\begin{lemma}\label{lemma_important}
    Assume that $1<p<q<N$, $\alpha>-p$ and $m\in\{p,q\}$. Then,
    \begin{equation}\label{important}
\left(
\int_{\mathbb{R}^N}|x|^\alpha |u|^{m_\alpha^*}\,\mathrm{d}x
\right)^{1/m_\alpha^*}
\leq
C\|\nabla u\|_{L^m(\mathbb{R}^N)}, \quad \forall \, u\in D_{\mathrm{rad}}^{1,m}(\mathbb{R}^N).
\end{equation} 
\end{lemma}

\begin{proof}
    Indeed, if \(\alpha=0\), then \eqref{important} is exactly the usual Sobolev inequality
$
D^{1,m}(\mathbb{R}^N)\hookrightarrow L^{m^*}(\mathbb{R}^N).
$
Now, assume that \(\alpha>0\). Then \(m_\alpha^*>m^*\). In view of radial lemma (see \cite[Lemma 1]{Su-Wang-Willem2}), there exists \(C_1>0\) such that
\begin{equation}\label{radial}
|u(x)|
\leq
C_1|x|^{-\frac{N-m}{m}}\|\nabla u\|_{L^m(\mathbb{R}^N)},
\qquad x\neq0.
\end{equation} 
It follows from \eqref{radial} and Sobolev embedding that
\[
\begin{aligned}
\int_{\mathbb{R}^N}|x|^\alpha |u|^{m_\alpha^*}\,\mathrm{d}x
&=
\int_{\mathbb{R}^N}
|x|^\alpha |u|^{m_\alpha^*-m^*}|u|^{m^*}\,\mathrm{d}x
\\
&\leq
C_1\|\nabla u\|_{L^m(\mathbb{R}^N)}^{m_\alpha^*-m^*}
\int_{\mathbb{R}^N}
|x|^{\alpha-\frac{N-m}{m}(m_\alpha^*-m^*)}
|u|^{m^*}\,\mathrm{d}x\\
& =
C_1\|\nabla u\|_{L^m(\mathbb{R}^N)}^{m_\alpha^*-m^*}
\int_{\mathbb{R}^N}|u|^{m^*}\,\mathrm{d}x\leq C_2\|\nabla u\|_{L^m(\mathbb{R}^N)}^{m_\alpha^*},
\end{aligned}
\]
where we used that \(\alpha-(N-m)(m_\alpha^*-m^*)/m=0\).

Now assume that \(-m<\alpha<0\). Set $\theta:=-\alpha/m\in(0,1)$. A direct computation gives
$$
m_\alpha^*
=
\theta m+(1-\theta)m^*.
$$
Thus, it follows from H\"older's inequality that
\[
\begin{aligned}
\int_{\mathbb{R}^N}|x|^\alpha |u|^{m_\alpha^*}\,\mathrm{d}x & = \int_{\mathbb{R}^N}\left(\frac{|u|^m}{|x|^m}\right)^\theta
\left(|u|^{m^*}\right)^{1-\theta}\,\mathrm{d}x\\
& \leq \left(
\int_{\mathbb{R}^N}\frac{|u|^m}{|x|^m}\,\mathrm{d}x
\right)^\theta
\left(
\int_{\mathbb{R}^N}|u|^{m^*}\,\mathrm{d}x
\right)^{1-\theta}.
\end{aligned}
\]
Using Hardy's inequality,
\[
\int_{\mathbb{R}^N}\frac{|u|^m}{|x|^m}\,\mathrm{d}x
\leq
C_3\int_{\mathbb{R}^N}|\nabla u|^m\,\mathrm{d}x,
\]
and the Sobolev inequality,
\[
\left(
\int_{\mathbb{R}^N}|u|^{m^*}\,\mathrm{d}x
\right)^{m/m^*}
\leq
C_4\int_{\mathbb{R}^N}|\nabla u|^m\,\mathrm{d}x,
\]
we conclude that
\[
\int_{\mathbb{R}^N}|x|^\alpha |u|^{m_\alpha^*}\,\mathrm{d}x
\leq
C_5\|\nabla u\|_{L^m(\mathbb{R}^N)}^{m\theta+m^*(1-\theta)}
\leq
C_5\|u\|_{m}^{m_\alpha^*}.
\]
Therefore, \eqref{important} holds for every \(\alpha>-m\). This finishes the proof.
\end{proof}

Now, we are able to prove the weighted Sobolev embedding in $\mathbb{R}^{N}$ stated in Theorem \ref{embeddingRN}.


\begin{proof}[\textbf{Proof of Theorem \ref{embeddingRN}}]
By applying \eqref{important} with \(m=p\) and \(m=q\), we have
\[
\|u\|_{L^{p_\alpha^*}(\mathbb{R}^N,|x|^\alpha)}
\leq
C_1\|\nabla u\|_{L^p(\mathbb{R}^N)}
\quad 
\mbox{and}
\quad 
\|u\|_{L^{q_\alpha^*}(\mathbb{R}^N,|x|^\alpha)}
\leq
C_1\|\nabla u\|_{L^q(\mathbb{R}^N)}.
\]
If \(s=p_\alpha^*\) or \(s=q_\alpha^*\) then the result follows directly from the estimates above. Case \(p_\alpha^*<s<q_\alpha^*\), we can use the interpolation inequality to obtain
\[
\begin{aligned} 
\|u\|_{L^s(\mathbb{R}^N,|x|^\alpha)}
& \leq
\|u\|_{L^{p_\alpha^*}(\mathbb{R}^N,|x|^\alpha)}^\theta
\|u\|_{L^{q_\alpha^*}(\mathbb{R}^N,|x|^\alpha)}^{1-\theta}\leq 
C_2
\|\nabla u\|_{L^p(\mathbb{R}^N)}^{\theta}
\|\nabla u\|_{L^q(\mathbb{R}^N)}^{1-\theta}\leq
C_2\|u\|_{p,q},
\end{aligned}
\]
for some \(\theta\in[0,1]\). Therefore,
$
E_{\mathrm{rad}}(\mathbb{R}^N)
\hookrightarrow
L^s(\mathbb{R}^N,|x|^\alpha)
$
continuously for every \(p_\alpha^*\leq s\leq q_\alpha^*\).

We now prove compactness.  Let \((u_n)_n\subset E_{\mathrm{rad}}(\mathbb{R}^N)\) be such that
$
u_n\rightharpoonup 0
$
weakly in $E_{\mathrm{rad}}(\mathbb{R}^N)$. Then \((u_n)_n\) is bounded in \(E_{\mathrm{rad}}(\mathbb{R}^N)\). Denote
$
M:=\sup_n\|u_n\|_{p,q}<\infty
$
and fix $s\in (p_\alpha^*,q_\alpha^*)$. For \(0<\rho<R\), we split \(\mathbb{R}^N\) into three regions:
\[
\mathbb{R}^N
=
B_\rho
\cup
(B_R\setminus B_\rho)
\cup
(\mathbb{R}^N\setminus B_R).
\]

First, we estimate the integral over the ball $B_\rho$. Since \(s<q_\alpha^*\), H\"older's inequality gives
\begin{equation}\label{B1}
\int_{B_\rho}|x|^\alpha |u_n|^s\,\mathrm{d}x
\leq
\left(
\int_{B_\rho}|x|^\alpha |u_n|^{q_\alpha^*}\,\mathrm{d}x
\right)^{s/q_\alpha^*}
\left(
\int_{B_\rho}|x|^\alpha\,\mathrm{d}x
\right)^{1-s/q_\alpha^*}.
\end{equation}
It follows from \eqref{important} with \(m=q\) that
\begin{equation}\label{B2}
\int_{B_\rho}|x|^\alpha |u_n|^{q_\alpha^*}\,\mathrm{d}x
\leq
C_3\|\nabla u_n\|_{L^q(\mathbb{R}^N)}^{q_\alpha^*}
\leq
C_3M^{q_\alpha^*}.
\end{equation}
Moreover, since \(\alpha>-p>-N\), we have
\begin{equation}
\int_{B_\rho}|x|^\alpha\,\mathrm{d}x
=
C_4\rho^{N+\alpha}.
\end{equation}
Thus, combining \eqref{B1}--\eqref{B3} we obtain
\begin{equation*}
\int_{B_\rho}|x|^\alpha |u_n|^s\,\mathrm{d}x
\leq
C_5M^s
\rho^{(N+\alpha)\left(1-\frac{s}{q_\alpha^*}\right)}.
\end{equation*}
Since \(s<q_\alpha^*\), given \(\varepsilon>0\), we can choose \(\rho>0\) small such that
\begin{equation}\label{B3}
    \int_{B_\rho}|x|^\alpha |u_n|^s\,\mathrm{d}x<\varepsilon.
\end{equation}

Next, we estimate over $\mathbb{R}^N\setminus B_R$. Since \(s>p_\alpha^*\), by using \eqref{radial} with \(m=p\), we deduce
\[
\begin{aligned}
\int_{\mathbb{R}^N\setminus B_R}|x|^\alpha |u_n|^s\,\mathrm{d}x
&=
\int_{\mathbb{R}^N\setminus B_R}
|x|^\alpha |u_n|^{s-p_\alpha^*}|u_n|^{p_\alpha^*}\,\mathrm{d}x
\\
&\leq
C_6M^{s-p_\alpha^*}
\int_{\mathbb{R}^N\setminus B_R}
|x|^{\alpha-(s-p_\alpha^*)\frac{N-p}{p}}
|u_n|^{p_\alpha^*}\,\mathrm{d}x.
\end{aligned}
\]
Set
$
\delta:=
(s-p_\alpha^*)(N-p)/p>0.
$
Then, for \(|x|\geq R\), we have
$
|x|^{\alpha-\delta}
\leq
R^{-\delta}|x|^\alpha.
$
Thus,
\[
\int_{\mathbb{R}^N\setminus B_R}|x|^\alpha |u_n|^s\,\mathrm{d}x
\leq
C_6M^{s-p_\alpha^*}R^{-\delta}
\int_{\mathbb{R}^N}
|x|^\alpha |u_n|^{p_\alpha^*}\,\mathrm{d}x.
\]
Using \eqref{important} with \(m=p\), we obtain
\begin{equation*}
\int_{\mathbb{R}^N\setminus B_R}|x|^\alpha |u_n|^s\,\mathrm{d}x
\leq
C_6M^s R^{-\delta}.
\end{equation*}
Since \(\delta>0\), for \(R>0\) sufficiently large, we can conclude that
\begin{equation}\label{RnBfinal}
    \int_{\mathbb{R}^N\setminus B_R}|x|^\alpha |u_n|^s\,\mathrm{d}x<\varepsilon.
\end{equation}

It remains to prove the convergence on the annulus
$
A_{\rho,R}:=B_R\setminus {B_\rho},
$ for $0<\rho<R$. Since \(|x|^\alpha\) is bounded on \(A_{\rho,R}\), it is enough to prove that
$
u_n\to0
$ strongly in $L^s(A_{\rho,R})$. We claim that \((u_n)_n\) is bounded in \(W^{1,q}(A_{\rho,R})\). In fact, by the Sobolev embedding
\(
D^{1,q}(\mathbb{R}^N)\hookrightarrow L^{q^*}(\mathbb{R}^N)
\)
and Hölder's inequality, since \(|A_{\rho,R}|<\infty\), we also have
\[
\|u_n\|_{L^q(A_{\rho,R})}\leq C_7M.
\]
By the compact embedding
\(
W^{1,q}(A_{\rho,R})\hookrightarrow L^q(A_{\rho,R}),
\)
up to a subsequence,
\(
u_n\to 0
\)
strongly in \(L^q(A_{\rho,R})\). Moreover, by \eqref{radial} and \(|x|\geq\rho\) on \(A_{\rho,R}\),
\[
\|u_n\|_{L^\infty(A_{\rho,R})}
\leq C_8\rho^{-\frac{N-q}{q}}\|u_n\|_{p,q}
\leq C_\rho M.
\]
Thus, since \(s> q\),
\[
\int_{A_{\rho,R}}|u_n|^s\,\mathrm{d}x
\leq
\|u_n\|_{L^\infty(A_{\rho,R})}^{s-q}
\int_{A_{\rho,R}}|u_n|^q\,\mathrm{d}x
\leq
(C_\rho M)^{s-q}\|u_n\|_{L^q(A_{\rho,R})}^q\to0,
\]
as \(n\to\infty\). Finally, since \(|x|^\alpha\) is bounded on \(A_{\rho,R}\), we conclude that
\begin{equation}\label{Afinal}
\int_{A_{\rho,R}}|x|^\alpha |u_n|^s\,\mathrm{d}x
\leq
C_{\rho,R,\alpha}
\int_{A_{\rho,R}}|u_n|^s\,\mathrm{d}x
\to0,
\end{equation}
as \(n\to\infty\). This combined with \eqref{B3} and \eqref{RnBfinal} implies that
\[
\int_{\mathbb{R}^N}|x|^\alpha |u_n|^s\,\mathrm{d}x
\to0, \quad \mbox{as } n\rightarrow\infty.
\]
This proves compactness and finishes the proof.
\end{proof}

\begin{corollary}\label{embedding-omega}
Let \(0<r<R\) be arbitrary, then the embedding $
E_{\mathrm{rad}}(\Omega)\hookrightarrow L^s(\Omega,|x|^\alpha)
$
is continuous for every
$
p_\alpha^*\leq s\leq q_\alpha^*,
$
in the following cases: \(\Omega=B_r\), \(\Omega=B_R\setminus B_r\), or \(\Omega=\mathbb{R}^N \setminus B_R\). Moreover, this embedding is compact whenever
$
p_\alpha^*<s<q_\alpha^*.
$
\end{corollary}
\begin{proof}
The result follows immediately by extending each function \(u\in E_{\mathrm{rad}}(\Omega)\) by zero outside \(\Omega\), observing that the extension belongs to \(E_{\mathrm{rad}}(\mathbb{R}^N)\) with the same \(\|\cdot\|_{p,q}\)-norm, and then applying Theorem \ref{embeddingRN}.
\end{proof}

\subsection{The weak formulation}

In this subsection we introduce the variational setting for the construction of nodal weak solutions of problem \eqref{eq1}.

Based on \cite{Bartsch-Willem} (see also \cite{Cao-Zhu,Fan-Zhao,Nehari}), our approach is to decompose $\mathbb{R}^{N}$ in radial domains and work on auxiliary problems. Precisely, for fixed $0 \leq r_{i} < r_{i + 1}\leq \infty$, we define the domain 
\begin{equation*}
    \Omega(r_{i},r_{i+1}) := \{x \in \mathbb{R}^{N} : r_{i} < \vert x \vert < r_{i+1}\}.
\end{equation*} 
Then, $\Omega(0,r_{i+1})$ is a ball (possibly $\mathbb{R}^{N})$, $\Omega(r_{i},r_{i+1})$ is an annulus (if $0<r_{i}<r_{i+1}<\infty$) and $\Omega(r_{i+1},\infty)$ is the exterior of a ball (if $r_{i+1}>0$). For simplicity, whenever the context is clear, we shall use the simplified notation \(\Omega\). We study the existence of positive and negative solutions for the Dirichlet problem
\begin{equation} \label{Sx}\tag{$\mathcal{P}_{\Omega}$}
\begin{cases}
- \Delta_pu - \Delta_qu =   |x|^\alpha |u|^{s-2}u, & \text{in } \Omega, \\
u  =  0, & \text{on } \partial \Omega.
\end{cases}
\end{equation}
The energy functional $I_{\Omega}:E_{\mathrm{rad}}(\Omega)\rightarrow\mathbb{R}$ associated with \eqref{Sx} is given by
\[ 
 I_\Omega(u)=\frac{1}{p}\|\nabla u\|_{L^p(\Omega)}^{p}+\frac{1}{q}\|\nabla u\|_{L^q(\Omega)}^{q}-\frac{1}{s}\int_{\Omega} |x|^{\alpha}|u|^{s}\,\mathrm{d}x. 
\]
In view of Corollary \ref{embedding-omega}, the functional \(I_\Omega\) is well defined and of class \(C^1\) on \(E_{\mathrm{rad}}(\Omega)\).
Furthermore, for every \(u,v\in E_{\mathrm{rad}}(\Omega)\), its derivative is given by
\begin{equation}\label{derivada-Gateaux}
    I_\Omega'(u)v  = \int_{\Omega}|\nabla u|^{p-2}\nabla u\nabla v\,\mathrm{d}x+\int_{\Omega}|\nabla u|^{q-2}\nabla u\nabla v\,\mathrm{d}x-\int_{\Omega}|x|^{\alpha}|u|^{s-2}u v\,\mathrm{d}x.
\end{equation}
We introduce the Nehari set
\begin{equation*}
    \mathcal{N}_\Omega = \{ u \in E_{\mathrm{rad}}(\Omega) \setminus \{0\} : I_\Omega'(u)u = 0 \},
\end{equation*}
and shall prove that the Nehari level
\[
c_{\mathcal{N}_\Omega}
:=
\inf_{u\in\mathcal{N}_\Omega} I_\Omega(u)
\]
is a critical value of $I_{\Omega}$.

\section{Nodal solutions via the Bartsch--Willem method}\label{sec3}

In this section, we adapt the method introduced by Bartsch and Willem \cite{Bartsch-Willem} in order to obtain nodal solutions for problem \eqref{eq1}.

\subsection{Existence of weak solutions}

In this subsection, we establish the existence of a nontrivial weak
solution of \eqref{Sx} by minimizing the energy functional
\(I_\Omega\) on the Nehari manifold \(\mathcal{N}_\Omega\).

We now prove that no sequence in \(\mathcal{N}_\Omega\) can converge
strongly to zero in \(E_{\mathrm{rad}}(\Omega)\).

\begin{lemma}\label{bounded-away-zero}
There exists $\beta>0$ such that
$
\|u\|_{p,q,\Omega}\geq \beta,$ for all $u\in \mathcal {N}_\Omega.
$
\end{lemma}

\begin{proof}
If $u\in \mathcal {N}_\Omega$, then it follows from Corollary \ref{embedding-omega} that
\begin{equation}\label{n1}
\|\nabla u\|_{L^p(\Omega)}^{p}
+
\|\nabla u\|_{L^q(\Omega)}^{q}
=
\int_{\Omega}|x|^\alpha |u|^s\dx\leq C_1\|u\|_{p,q,\Omega}^s\leq C_2\left(\|\nabla u\|_{L^p(\Omega)}^{s} + \|\nabla u\|_{L^q(\Omega)}^{s}\right).
\end{equation}
Suppose, by contradiction, that there exists a sequence \((u_n)_n\subset\mathcal{N}_\Omega\) such that
$
\|u_n\|_{p,q,\Omega}\to 0
$
as $n\rightarrow\infty$. From \eqref{n1},
\[
1\leq C_2\left(\|\nabla u_n\|_{L^p(\Omega)}^{s-p} + \|\nabla u_n\|_{L^q(\Omega)}^{s-q}\right).
\]
Since $
\|u_n\|_{p,q,\Omega}\to 0,
$
as $n\rightarrow\infty$, then \(\|\nabla u_n\|_{L^p(\Omega)},\|\nabla u_n\|_{L^q(\Omega)}\to0\),
as $n\rightarrow\infty$. Consequently \(1\leq0\), because \(p<q<s\), which leads us to a contradiction. This proves the lemma.
\end{proof}

In the next result, we prove that \(\mathcal{N}_\Omega\) is nonempty
by showing that every nonzero element of \(E_{\mathrm{rad}}(\Omega)\)
can be projected onto \(\mathcal{N}_\Omega\).

\begin{lemma}\label{lemaproj} For each $u \in E_{\mathrm{rad}}(\Omega) \setminus \{0\},$ there exists a unique $t_u >0$ such that $t_uu \in \mathcal{N}_\Omega$ and
\[
I_\Omega(t_uu) = \max_{t\geq 0} I_\Omega(tu).
\]
Furthermore, if \(I'_\Omega(u)u\leq 0\), then \(0<t_u\leq 1\).
\end{lemma}

\begin{proof}
    Let \(u\in E_{\mathrm{rad}}(\Omega)\setminus\{0\}\) be fixed and define $h_{u}(t) = I_\Omega(tu)$ in $[0,\infty)$. By definition, $tu \in \mathcal{N}_\Omega$ if and only if $h_{u}'(t) = 0$. From \eqref{derivada-Gateaux},
\begin{equation} \label{eq:6art}
        h'_{u}(t) = I_\Omega'(tu)u = t^{p-1}\int_{\Omega}|\nabla u|^{p}\,\mathrm{d}x+t^{q-1}\int_{\Omega}|\nabla u|^{q}\,\mathrm{d}x-t^{s-1}\int_{\Omega}|x|^{\alpha}|u|^{s}\,\mathrm{d}x.
    \end{equation}
Note that $h'_{u}(0) = 0,$ $h'_{u}(t) > 0,$ for small $t > 0$ and $h'_{u}(t) < 0,$ for $t$ large. Therefore, there exists at least one \(t_u>0\) such that $h_u'(t_u)=0$.

To prove the uniqueness, observe that \eqref{eq:6art} yields
\[
\frac{h'_{u}(t)}{t^{q-1}}=t^{p-q}\int_{\Omega}|\nabla u|^{p}\,\mathrm{d}x+\int_{\Omega}|\nabla u|^{q}\,\mathrm{d}x-t^{s-q}\int_{\Omega}|x|^{\alpha}|u|^{s}\,\mathrm{d}x.
\]
    Since $p<q<s$, we conclude that the map $t\mapsto h'_{u}(t)/t^{q-1}$ is decreasing. Now, suppose that there exists $0<t_1<t_2$, such that $h_{u}'(t_1)=h_{u}'(t_2)=0$. Thus,  
$$
0=\frac{h'_{u}(t_2)}{t_2^{q-1}}<\frac{h'_{u}(t_1)}{t_{1}^{q-1}} = 0,
$$ 
which is a contradiction. Therefore, $\max_{t\geq0} h_{u}(t)$ is achieved at a unique $t = t_u>0$ and $t_u u \in \mathcal{N}_\Omega,$ because $h'_{u}(t_u) = 0.$

Finally, using that $h_{u}' > 0$ near the origin and it has a unique critical point, one concludes that $h_{u}' > 0$ in $(0,t_{u})$ and
$h_{u}' < 0$ in $(t_{u},\infty)$. Consequently, 
$t_{u}\in(0,1]$, whenever
$h_{u}'(1)=I_{\Omega}'(u)u\leq 0$, completing the proof.
\end{proof}

By Lemma \ref{lemaproj}, we can conclude that \(\mathcal{N}_\Omega \neq \emptyset\). Moreover, given \(u\in\mathcal{N}_\Omega\), one has
\begin{equation}\label{limitada}
    I_\Omega(u) = I_\Omega(u) - \frac{1}{s}I'_\Omega(u)u = \left(\frac{1}{p}-\frac{1}{s}\right)\int_\Omega|\nabla u|^p\dx + \left(\frac{1}{q}-\frac{1}{s}\right)\int_\Omega|\nabla u|^q\dx>0,
\end{equation}
where we used \(p<q<s\). This implies that the value
\begin{equation*}
    c_{\mathcal{N}_\Omega} = c_{\mathcal{N}_\Omega}(r_{i},r_{i+1}) := \inf_{u\in \mathcal{N}_\Omega}I_\Omega(u)
\end{equation*}
belongs to \([0,\infty)\).

We next prove that every minimizer of \(I_\Omega\) on the Nehari
manifold \(\mathcal{N}_\Omega\) is a critical point of \(I_\Omega\).

\begin{lemma}
    It holds that $\mathcal{N}_{\Omega}$ is a $C^{1}$-manifold. Moreover, if $v\in\mathcal{N}_\Omega$ is such that
\[
I_{\Omega}(v)
=
\inf_{u\in\mathcal{N}_{\Omega}} I_{\Omega}(u),
\]
then $I_{\Omega}'(v)=0$.
\end{lemma}

\begin{proof}
    Set
\(
J_{\Omega}(u):=I_{\Omega}'(u)u
\) and see that \(\mathcal{N}_{\Omega}=J_{\Omega}^{-1}(0)\). Note that
\[
J_\Omega(u)
=
\int_\Omega |\nabla u|^p\dx
+
\int_\Omega |\nabla u|^q\dx
-
\int_\Omega |x|^\alpha |u|^s\dx.
\]
Moreover,
\[
J_\Omega'(u)v
=
p\int_\Omega |\nabla u|^{p-2}\nabla u\nabla v\dx+
q\int_\Omega |\nabla u|^{q-2}\nabla u\nabla v\dx-
s\int_\Omega |x|^\alpha |u|^{s-2}uv\dx.
\]
In particular,
\[
J_\Omega'(u)u
=
p\int_\Omega |\nabla u|^p\dx
+
q\int_\Omega |\nabla u|^q\dx
-
s\int_\Omega |x|^\alpha |u|^s\dx.
\]
Given \(u\in\mathcal N_\Omega\), we have that
\[
J_\Omega'(u)u
=
(p-s)\int_\Omega |\nabla u|^p\dx
+
(q-s)\int_\Omega |\nabla u|^q\dx<0,
\]
because $p<q<s$. Therefore, \(0\) is a regular value of \(J_{\Omega}\) and hence
\(\mathcal{N}_{\Omega}\) is a \(C^{1}\)-manifold.

To prove the second part of result, suppose that 
\(v\in\mathcal{N}_{\Omega}\) is such that
\[
I_{\Omega}(v)
=
\inf_{u\in\mathcal{N}_{\Omega}} I_{\Omega}(u).
\]
Thus, applying the Lagrange multipliers theorem, there exists \(\mu\in\mathbb{R}\) such that
\[
I_{\Omega}'(v)=\mu J_{\Omega}'(v),
\]
Hence,
\[
0
=
I_{\Omega}'(v)v
=
\mu J_{\Omega}'(v)v,
\]
which implies that \(\mu=0\), because \(J_\Omega'(v)v<0\). As consequence, 
\(
I_{\Omega}'(v)=0,
\)
which finishes the proof.
\end{proof}

Finally, we show that the infimum of \(I_\Omega\) over
\(\mathcal{N}_\Omega\) is attained.

\begin{lemma}
    There exists \(u_0 \in \mathcal{N}_\Omega\) such that 
    \[
I_{\Omega}(u_0)
=c_{\mathcal{N}_\Omega}=
\inf_{u\in\mathcal{N}_{\Omega}} I_{\Omega}(u).
\]
\end{lemma}

\begin{proof} Let $(u_n)_n \subset \mathcal{N}_\Omega$ be a minimizing sequence for
$c_{\mathcal{N}_\Omega}$. The estimate \eqref{limitada} yields
\[
c_{\mathcal{N}_\Omega}+o_n(1) \geq C_1\left(\int_\Omega|\nabla u_n|^p\dx+\int_\Omega|\nabla u_n|^q\dx\right),
\]
implying that the sequence \((u_n)_n\) is bounded in \(E_{\mathrm{rad}}(\Omega)\). So, up to a subsequence, $u_n \rightharpoonup u_0$
weakly in $E_{\mathrm{rad}}(\Omega)$.

We claim that \(u_0\neq 0\). Otherwise, from Corollary \ref{embedding-omega} it follows that
\[
0\leq\int_{\Omega} |\nabla u_n|^p\dx\leq\int_{\Omega} |\nabla u_n|^p\dx
+ \int_{\Omega} |\nabla u_n|^q\dx = \int_\Omega|x|^\alpha|u_n|^s\dx=o_n(1).
\]
Similarly,
\[
\int_{\Omega} |\nabla u_n|^q\dx=o_n(1),
\]
and hence \(\|u_n\|_{p,q,\Omega}=o_n(1)\). From Lemma \ref{bounded-away-zero},
\[
0<\beta\leq \|u_n\|_{p,q,\Omega}=o_n(1),
\]
which is a contradiction, and hence \(u_0\neq 0\).

From Lemma \ref{lemaproj}, there exists \(t_{u_0}>0\) such that \(t_{u_0}u_0\in\mathcal{N}_\Omega\). Since the function
$s\mapsto s^\theta$ is convex for $\theta>1$, we have that
\begin{equation}\label{lim-infimo}
    \|\nabla u_0\|_{L^\theta(\Omega)}^\theta \leq \liminf_{n\to\infty}\|\nabla u_n\|_{L^\theta(\Omega)}^\theta,
\end{equation}
for $\theta\in\{p,q\}$. This, together with Corollary \ref{embedding-omega} implies that 
\[
I_\Omega'(u_0)u_0
\leq
\liminf_{n\to\infty} I_\Omega'(u_n)u_n
=0.
\]
So, from Lemma \ref{lemaproj} we get \(0<t_{u_0}\leq 1\). This condition in \(t_{u_0}\), \eqref{limitada}, and \eqref{lim-infimo} infer that
\[
c_{\mathcal{N}_\Omega} \leq I_\Omega(t_{u_0}u_0) \leq \liminf_{n\to\infty} I_\Omega(u_n)=c_{\mathcal{N}_\Omega}.
\]
As consequence, \(t_{u_0}=1\), because if \(t_{u_0}<1\), we would have to \(c_{\mathcal{N}_\Omega}<c_{\mathcal{N}_\Omega}\), resulting in a contradiction. Thus,
\[
I_\Omega(u_0)=I_\Omega(t_{u_0}u_0)= c_{\mathcal{N}_\Omega}.
\]
The proof is completed.
\end{proof}

The following result is essential for the subsequent proof of a version
of the principle of symmetric criticality adapted to our setting.

\begin{lemma}\label{continuidade-termo-henon}
Assume that
$
1<p<q<N,$ $\alpha>-p$,
and
$
p_\alpha^*<s<q_\alpha^*.
$
Let \(\Omega\) be either \(\mathbb{R}^N\), a ball, an annulus, or an exterior domain centered at the origin. If \(u\in E_{\mathrm{rad}}(\Omega)\), then there exists a constant \(D=D(u)>0\) such that
\[
\left|
\int_\Omega |x|^\alpha |u|^{s-2}u w\,\mathrm{d}x
\right|
\leq
D\|w\|_{p,q,\Omega},\quad\forall \, w\in E(\Omega).
\]    
\end{lemma}

\begin{proof}
    We first prove the result for \(\Omega=\mathbb{R}^N\). Given \(u\in E_{\mathrm{rad}}(\mathbb{R}^N)\) and \(w\in E(\mathbb{R}^N)\), we write
    \begin{equation}\label{integral-soma}
\int_{\mathbb{R}^N}|x|^\alpha |u|^{s-1}|w|\,\mathrm{d}x
=
\int_{B_1}|x|^\alpha |u|^{s-1}|w|\,\mathrm{d}x
+
\int_{\mathbb{R}^N\setminus B_1}|x|^\alpha |u|^{s-1}|w|\,\mathrm{d}x.        
    \end{equation}
Since \(E(\mathbb{R}^N)\subset D^{1,m}(\mathbb{R}^N)\), for \(m\in\{p,q\}\), we can apply \eqref{radial} (replacing \(p\) by \(q\)), to get
    \[
\int_{B_1\setminus B_\varepsilon}|x|^\alpha |u|^{s-1}|w|\,\mathrm{d}x \leq C_1\|u\|_{p,q}^{s-1}\int_{B_1}|x|^{\alpha+1-(s-1)(N-q)/q}\frac{|w|}{|x|}\dx,
\]
where \(0<\varepsilon<1\). Next, we use H\"older's inequality together with the following Hardy's inequality
\[
\left(
\int_{\mathbb{R}^N}
\frac{|w|^q}{|x|^q}\,\mathrm{d}x
\right)^{1/q}
\leq
C_2
\|\nabla w\|_{L^q(\mathbb{R}^N)}
\]
to obtain
\[
\int_{B_1\setminus B_\varepsilon}|x|^\alpha |u|^{s-1}|w|\,\mathrm{d}x \leq C_3\|u\|_{p,q}^{s-1}\left(\int_{B_1}|x|^\xi \dx\right)^{(q-1)/q}\|w\|_{p,q},
\]
where
\[
\xi:=\frac{q}{q-1}\left(\alpha+1-(s-1)\frac{(N-q)}{q}\right).
\]
Noting that
\[
\begin{aligned}
    \xi+N>0 \quad &\Longleftrightarrow \quad q\alpha + q -(s-1)(N-q)+N(q-1)>0
    \\
    & \Longleftrightarrow \quad q(N+\alpha) -(N-q) >(s-1)(N-q)
    \\
    & \Longleftrightarrow \quad q_\alpha^* -1 >s-1,
\end{aligned}
\]
it holds that \(|x|^\xi\in L^1(B_1)\). So,
\[
\int_{B_1\setminus B_\varepsilon}|x|^\alpha |u|^{s-1}|w|\,\mathrm{d}x \leq C_4\|w\|_{p,q},
\]
for some \(C_4=C_4(u)>0\). Taking \(\varepsilon\to 0^+\), one concludes that
\begin{equation}\label{integral-B1}
    \int_{B_1}|x|^\alpha |u|^{s-1}|w|\,\mathrm{d}x \leq C_4\|w\|_{p,q}.
\end{equation}

Now we estimate the integral over \(\mathbb{R}^N\setminus B_1\). Again by \eqref{radial},
\[
\int_{\mathbb{R}^N\setminus B_1}
|x|^\alpha |u|^{s-1}|w|\,\mathrm{d}x \leq C_5\|u\|_{p,q}^{s-1}\int_{\mathbb{R}^N\setminus B_1}
|x|^{\alpha+1-(s-1)(N-p)/p}\frac{|w|}{|x|}\dx.
\]
Following the same argument as above and using that \(s>p_\alpha^*\), we obtain
\[
\int_{\mathbb{R}^N\setminus B_1}
|x|^\alpha |u|^{s-1}|w|\,\mathrm{d}x \leq C_6\|w\|_{p,q},
\]
for some \(C_6=C_6(u)>0\). Combining this inequality with \eqref{integral-soma} and \eqref{integral-B1}, we have that
\[
\left|
\int_{\mathbb{R}^N} |x|^\alpha |u|^{s-2}u w\,\mathrm{d}x
\right|
\leq
(C_4+C_6)\|w\|_{p,q},\quad\forall \, w\in E(\mathbb{R}^N).
\]

Finally, if \(\Omega\) is a ball, an annulus or an exterior domain centered at the origin, we extend \(u\) and \(w\) by zero outside \(\Omega\) and the proof is finalized.
\end{proof}

Since we seek radial solutions, we look for critical points of
\(I_\Omega\) in the radial space \(E_{\mathrm{rad}}(\Omega)\).
To ensure that these critical points are weak solutions of
\eqref{Sx}, we establish the following version of the
principle of symmetric criticality, based on \cite{Palais}.

\begin{proposition}\label{palais}
    Let \(\Omega\) be either \(\mathbb{R}^N\), a ball, an annulus, or an exterior domain centered at the origin. If \(u\in E_{\mathrm{rad}}(\Omega)\) is a critical point of \(I_\Omega\), then \(u\) is a weak solution of \eqref{Sx}.
\end{proposition}

\begin{proof} Fix $u\in E_{\mathrm{rad}}(\Omega)$ and define the linear functional
\(
T_u:E(\Omega)\rightarrow \mathbb{R}
\)
by
\[
T_u(w)
:=
\int_{\Omega}|\nabla u|^{p-2}\nabla u\nabla w\dx
+
\int_{\Omega}|\nabla u|^{q-2}\nabla u\nabla w\dx
-
\int_{\Omega}|x|^\alpha |u|^{s-2}uw\dx.
\]
We first verify that $T_u$ is continuous on $E(\Omega)$. Indeed, by H\"older's inequality, we have
\[
\left|
\int_{\Omega}|\nabla u|^{p-2}\nabla u\nabla w\dx
\right|
\leq
\|\nabla u\|_{L^p(\Omega)}^{p-1}
\|\nabla w\|_{L^p(\Omega)}
\leq
C_1\|w\|_{p,q,\Omega}
\]
and
\[
\left|
\int_{\Omega}
|\nabla u|^{q-2}\nabla u\nabla w\dx
\right|
\leq
\|\nabla u\|_{L^q(\Omega)}^{q-1}
\|\nabla w\|_{L^q(\Omega)}
\leq
C_2\|w\|_{p,q,\Omega}.
\]
These estimates combined with Lemma \ref{continuidade-termo-henon} infer that
\[
|T_u(w)| \leq C_3\|w\|_{p,q,\Omega}, \quad \forall \, w\in E(\Omega).
\]
So, $T_u$ is continuous on $E(\Omega)$.
Since $E(\Omega)$ is uniformly convex, in view of \cite[Lemme 14.5]{Willem2}, there exists a unique $\bar{u}\in E(\Omega)$, with
\begin{equation}\label{unicidade-norma}
\|\bar{u}\|_{p,q,\Omega}=1 \quad \mbox{and} \quad T_u(\bar{u})
=
\max_{\|v\|_{p,q,\Omega}=1}
|T_u(v)|.
\end{equation}
Recall that, for $g\in\mathcal{O}(N)$ (\(\mathcal{O}(N)\) denotes the group of orthogonal transformations in $\mathbb{R}^N$), the action of $g$ on a function $v$ is defined by
\[
(gv)(x):=v(g^{-1}x).
\]
For every $g\in\mathcal{O}(N)$, the invariance under orthogonal transformations, together with
\[
|g^{-1}x|=|x|
\quad\text{and}\quad
g^{-1}(\Omega)=\Omega,
\]
shows, after the change of variables $y=g^{-1}x$, that
\[
T_u(g\bar{u})=T_u(\bar{u}).
\]
The same argument yields
\[
\|g\bar{u}\|_{p,q,\Omega}=\|\bar{u}\|_{p,q,\Omega}=1.
\]
Hence, the uniqueness property established in \eqref{unicidade-norma} implies
\[
g\bar{u}=\bar{u},
\quad \forall \, g\in\mathcal{O}(N),
\]
and so \(\bar{u}\in E_{\mathrm{rad}}(\Omega)\). Since $u$ is a critical point of the restriction of $I_\Omega$ to
$E_{\mathrm{rad}}(\Omega)$, it follows that
$
T_u(\bar{u})=0.
$
Thus,
\[
\max_{\|v\|_{p,q,\Omega}=1}|T_u(v)|
=
T_u(\bar{u})
=
0.
\]

Now, given any $w\in E(\Omega)$, we may write
\[
T_u(w)
=
\|w\|_{p,q,\Omega}
T_u\left(\frac{w}{\|w\|_{p,q,\Omega}}\right)
=
0.
\]
Therefore,
\[
\int_{\Omega}|\nabla u|^{p-2}\nabla u\nabla w\dx
+
\int_{\Omega}|\nabla u|^{q-2}\nabla u\nabla w\dx
-
\int_{\Omega}|x|^\alpha |u|^{s-2}uw\dx=0, \quad \forall \, w \in E(\Omega),
\]
 which proves that $u$ is a weak solution of \eqref{Sx}.
\end{proof}

\subsection{Regularity}

In this subsection, we establish regularity and positivity results
for nontrivial nonnegative weak solutions.

The first result establishes the boundedness of weak solutions.

\begin{lemma}\label{lema-iteracao-de-moser}
    If \(u\in E_\mathrm{rad}(\mathbb{R}^N)\) is a nonnegative weak solution of \eqref{eq1}, then \(u\in L^\infty(\mathbb{R}^N)\).
\end{lemma}

\begin{proof}
For each $m \in \mathbb{N}$ and $\gamma>1$, we consider the sets
\[
A_m := \left\{x \in \mathbb{R}^N : u^{\gamma-1}(x) \leq m\right\} \quad \mbox{and} \quad B_m := \mathbb{R}^N \setminus A_m.
\]
Next, define
\[
u_m :=
\begin{cases}
u^{q(\gamma-1)+1}, & \text{in } A_m,\\
m^q u, & \text{in } B_m,
\end{cases}
\]
and see that
\[
\nabla u_m =
\begin{cases}
(q\gamma-q+1)u^{q(\gamma-1)}\nabla u,
& \text{in } A_m,\\
m^q\nabla u,
& \text{in } B_m.
\end{cases}
\]

We claim that $u_m \in E_\mathrm{rad}(\mathbb{R}^N)$, for any $m \in \mathbb{N}$. In fact, we can estimate
\[
\begin{aligned}
\int_{\mathbb{R}^N} |\nabla u_m|^q\,\mathrm{d}x
&=
(q\gamma-q+1)^q
\int_{A_m}u^{q^2(\gamma-1)}|\nabla u|^q\,\mathrm{d}x
+
m^{q^2}\int_{B_m}|\nabla u|^q\,\mathrm{d}x
\\
&\leq
(q\gamma-q+1)^qm^{q^2}
\int_{A_m}|\nabla u|^q\,\mathrm{d}x
+
m^{q^2}\int_{B_m}|\nabla u|^q\,\mathrm{d}x
<\infty.
\end{aligned}
\]
Similarly, $|\nabla u_m| \in L^p(\mathbb{R}^N)$. So, $u_m \in E_\mathrm{rad}(\mathbb{R}^N)$ and the claim is verified.

Now, using that $u$ is solution of \eqref{eq1}, one concludes that
\begin{equation}\label{u-solucao-um}
    \int_{\mathbb{R}^N}
|\nabla u|^{p-2}\nabla u\nabla u_m\,\mathrm{d}x
+
\int_{\mathbb{R}^N}
|\nabla u|^{q-2}\nabla u\nabla u_m\,\mathrm{d}x
=
\int_{\mathbb{R}^N}|x|^\alpha |u|^{s-2}uu_m\,\mathrm{d}x.
\end{equation}

Next, define
\[
z_m :=
\begin{cases}
u^\gamma, & \text{in } A_m,\\
mu, & \text{in } B_m.
\end{cases}
\]
Consequently,
\[
\nabla z_m =
\begin{cases}
\gamma u^{\gamma-1}\nabla u, & \text{in } A_m,\\
m\nabla u, & \text{in } B_m.
\end{cases}
\]
Using the same idea as that applied to \(u_m\), one has \(z_m\in E_\mathrm{rad}(\mathbb{R}^N)\), for any \(m\in\mathbb{N}\). See that
\[
\int_{\mathbb{R}^N}|\nabla z_m|^q\dx
=
\gamma^q\int_{A_m}u^{q(\gamma-1)}|\nabla u|^q\dx
+
m^q\int_{B_m}|\nabla u|^q\dx
\]
and
\begin{equation}\label{gradiente-u-um}
    \int_{\mathbb{R}^N}|\nabla u|^{q-2}\nabla u\nabla u_m\dx
=
(q\gamma-q+1)
\int_{A_m}u^{q(\gamma-1)}|\nabla u|^q\dx
+
m^q\int_{B_m}|\nabla u|^q\dx.
\end{equation}
So,
\[
\int_{\mathbb{R}^N}|\nabla z_m|^q\dx
-
\int_{\mathbb{R}^N}|\nabla u|^{q-2}\nabla u\nabla u_m\dx
=
(\gamma^q-q\gamma+q-1)
\int_{A_m}u^{q(\gamma-1)}|\nabla u|^q\dx.
\]
By \eqref{gradiente-u-um} we have that
\[
\int_{A_m} u^{q(\gamma-1)}|\nabla u|^q\dx
\leq
\frac{1}{(q\gamma-q+1)}
\int_{\mathbb{R}^N}|\nabla u|^{q-2}\nabla u\nabla u_m\dx.
\]
Hence,
\[
\begin{aligned}
\int_{\mathbb{R}^N}|\nabla z_m|^q\dx
&\leq
\frac{\gamma^q-q\gamma+q-1}{q\gamma-q+1}
\int_{\mathbb{R}^N}|\nabla u|^{q-2}\nabla u\nabla u_m\dx
+
\int_{\mathbb{R}^N}|\nabla u|^{q-2}\nabla u\nabla u_m\dx\\
&=
\frac{\gamma^q}{q\gamma-q+1}
\int_{\mathbb{R}^N}|\nabla u|^{q-2}\nabla u\nabla u_m\dx\leq
\gamma^q
\int_{\mathbb{R}^N}|\nabla u|^{q-2}\nabla u\nabla u_m\dx.
\end{aligned}
\]
Combining this inequality with \eqref{u-solucao-um}, one deduces
\[
\int_{\mathbb{R}^N}|\nabla z_m|^q\dx
\leq
\gamma^q
\left(
-\int_{\mathbb{R}^N}|\nabla u|^{p-2}\nabla u\nabla u_m\dx
+
\int_{\mathbb{R}^N}|x|^\alpha|u|^{s-2}uu_m\dx
\right).
\]
Since
\[
-\int_{\mathbb{R}^N}|\nabla u|^{p-2}\nabla u\nabla u_m\dx
=
-(q\gamma-q+1)
\int_{A_m}u^{q(\gamma-1)}|\nabla u|^{p}\dx
-
m^q\int_{B_m}|\nabla u|^p\dx
\leq 0,
\]
we can conclude that
\[
\int_{\mathbb{R}^N}|\nabla z_m|^q\dx
\leq \gamma^q\int_{\mathbb{R}^N}|x|^\alpha|u|^{s-2}uu_m\dx.
\]
Since \(u_m\leq u^{q\gamma-q+1}\) in \(\mathbb{R}^N\), we can estimate
\[
\int_{\mathbb{R}^N}|\nabla z_m|^q\dx
\leq \gamma^q\int_{\mathbb{R}^N}|x|^\alpha u^{s+q\gamma-q}\dx=\gamma^q\int_{\mathbb{R}^N}|x|^\alpha u^{s-q}u^{q\gamma}\dx.
\]
Applying Hölder's inequality with \(1/t_1+1/t_2=1\) and 
\[
\frac{s}{s-q}<t_1<\frac{q_\alpha^*}{s-q}
\]
we have that
\[
\begin{aligned}
    \int_{\mathbb{R}^N}|\nabla z_m|^q\dx
&\leq \gamma^q\left(\int_{\mathbb{R}^N}|x|^\alpha u^{t_1(s-q)}\dx\right)^{1/t_1}\left(\int_{\mathbb{R}^N}|x|^\alpha u^{t_2q\gamma}\dx\right)^{1/t_2}\\\
&\leq 
C_1\gamma^q\|u\|_{L^{t_2q\gamma}(\mathbb{R}^N,|x|^\alpha)}^{q\gamma},
\end{aligned}
\]
for some \(C_1=C_1(u)>0\), where we apply Theorem \ref{embeddingRN}. Next, from Lemma \ref{important},
\[
\|z_m\|_{L^{q_\alpha^*}(\mathbb{R}^N,|x|^\alpha)}^q\le C_2\gamma^q\|u\|_{L^{t_2q\gamma}(\mathbb{R}^N,|x|^\alpha)}^{q\gamma}. 
\]
Recalling that \(|z_m|=u^\gamma\) in \(A_m\), one has
\[
\|u\|_{L^{q_\alpha^*\gamma}(\mathbb{R}^N,|x|^\alpha)}^{q\gamma}\le C_2\gamma^q\|u\|_{L^{t_2q\gamma}(\mathbb{R}^N,|x|^\alpha)}^{q\gamma},
\]
where we used Fatou's Lemma. So,
\begin{equation}\label{iteracao-de-moser}
    \|u\|_{L^{q_\alpha^*\gamma}(\mathbb{R}^N,|x|^\alpha)}\le C_2^{1/q\gamma}\gamma^{1/\gamma}\|u\|_{L^{t_2q\gamma}(\mathbb{R}^N,|x|^\alpha)}.
\end{equation}

We claim that \(q_\alpha^*/t_2q>1\). In fact, since \(s/(s-q)<t_1\), then \(q/(s-q)<t_1-1\), and hence
\[
t_2=\frac{t_1}{t_1-1}=t_1\cdot\frac{1}{t_1-1}<\frac{q_\alpha^*}{s-q}\cdot \frac{s-q}{q} = \frac{q_\alpha^*}{q}.
\]
Next, we start Moser's iteration, choosing $\gamma=\sigma=q_\alpha^*/t_2q>1$. Consequently,
\[
\|u\|_{L^{q_\alpha^*\sigma}(\mathbb{R}^N,|x|^\alpha)}\le C_2^{1/q\sigma}\sigma^{1/\sigma}\|u\|_{L^{q_\alpha^*}(\mathbb{R}^N,|x|^\alpha)}.
\]
Now, choosing \(\gamma=\sigma^2\) in \eqref{iteracao-de-moser} and observing that \(t_2 q \sigma^2=(t_2q\sigma)\sigma=q_\alpha^*\sigma\), we get
\[
\|u\|_{L^{q_\alpha^*\sigma^2}(\mathbb{R}^N,|x|^\alpha)}\le C_2^{1/q\sigma^2}(\sigma^2)^{1/\sigma^2}\|u\|_{L^{q_\alpha^*\sigma}(\mathbb{R}^N,|x|^\alpha)}.
\]
Combining the last two inequalities,
\[
\begin{aligned}
    \|u\|_{L^{q_\alpha^*\sigma^2}(\mathbb{R}^N,|x|^\alpha)}&\le C_2^{1/q\sigma^2}(\sigma^2)^{1/\sigma^2} \left(C_2^{1/q\sigma}\sigma^{1/\sigma}\|u\|_{L^{q_\alpha^*}(\mathbb{R}^N,|x|^\alpha)}\right) 
    \\
    &= C_2^{1/q\sigma+1/q\sigma^2}\sigma^{1/\sigma+2/\sigma^2}\|u\|_{L^{q_\alpha^*}(\mathbb{R}^N,|x|^\alpha)}.
\end{aligned}
\]
Consequently, we can conclude that
\begin{equation}\label{des-final-iteracao-moser}
    \|u\|_{L^{q_\alpha^*\sigma^j}(\mathbb{R}^N,|x|^\alpha)} \le 
C_2^{\sum_{j=1}^{\infty}1/q\sigma^j}\sigma^{\sum_{j=1}^{\infty}j/\sigma^j}\|u\|_{L^{q_\alpha^*}(\mathbb{R}^N,|x|^\alpha)}.
\end{equation}
Since
\[
\sum_{j=1}^{\infty}\frac{1}{q\sigma^j}<\infty
\quad\text{and}\quad
\sum_{j=1}^{\infty}\frac{j}{\sigma^j}<\infty,
\]
we get
\[
\|u\|_{L^{q_\alpha^*\sigma^j}(\mathbb{R}^N,|x|^\alpha)} \le  C_3\|u\|_{L^{q_\alpha^*}(\mathbb{R}^N,|x|^\alpha)}.
\]

Now, consider $R>0$. If \(-p<\alpha\leq 0\), then 
\[
\int_{B_R}u^{q_\alpha^*\sigma^j}\dx \leq \frac{1}{R^\alpha}\int_{B_R}|x|^\alpha u^{q_\alpha^*\sigma^j}\dx.
\]
Combining the last two estimates, we have that
\[
\|u\|_{L ^{q_\alpha^*\sigma^j}(B_R)} \leq \frac{C_3}{R^{\alpha/(q_\alpha^*\sigma^j)}}\|u\|_{L^{q_\alpha^*}(\mathbb{R}^N,|x|^\alpha)}.
\]
Taking the limit as $j\to\infty$, follows that $u\in L^\infty(B_R)$. Case \(\alpha>0\), we can apply Hölder's inequality to get 
\[
\begin{aligned}
\int_{B_R}u^{\sigma^j}\dx
&=
\int_{B_R}
|x|^{-\alpha/q_\alpha^*}
|x|^{\alpha/q_\alpha^*}u^{\sigma^j}\dx\\
&\le
\left(
\int_{B_R}|x|^{-\alpha/(q_\alpha^*-1)}\dx
\right)^{(q_\alpha^*-1)/q_\alpha^*}
\left(
\int_{B_R}|x|^\alpha u^{q_\alpha^*\sigma^j}\dx
\right)^{1/q_\alpha^*}\\
&=
C_4
\left(
\int_{B_R}|x|^\alpha u^{q_\alpha^*\sigma^j}\dx
\right)^{1/q_\alpha^*},
\end{aligned}
\]
where
\[
C_4:=
\left(
\int_{B_R}|x|^{-\alpha/(q_\alpha^*-1)}\dx
\right)^{(q_\alpha^*-1)/q_\alpha^*}.
\]
Since \(1<q<N\) and \(\alpha>0\), it holds that \(-\alpha/(q_\alpha^*-1)+ N>0\). So, \(|x|^{-\alpha/(q_\alpha^*-1)}\in L^1(B_R)\), which ensures that \(C_4<\infty\). The above inequality and \eqref{des-final-iteracao-moser} infer that
\[
\|u\|_{L^{\sigma^j}(B_R)} \leq C_4^{1/\sigma^j}C_3\|u\|_{L^{q_\alpha^*}(\mathbb{R}^N,|x|^\alpha)}.
\]
Again taking the limit as $j\to\infty$, follows that $u\in L^\infty(B_R)$. 

Using the fact that $u\in L^\infty(B_R)$ together with  \eqref{radial}, one can conclude that \(u \in L^\infty(\mathbb{R}^N)\), completing the proof.
\end{proof}

\begin{remark}\label{rem:Linfty-other-domains} The same argument used in Lemma \ref{lema-iteracao-de-moser} can be applied to \(\Omega=B_r\), \(\Omega=B_R\setminus B_r\), or \(\Omega=\mathbb{R}^N \setminus B_R\) to get \(L^\infty\)-regularity, because the weighted Sobolev embeddings used in the iteration hold in
\(E_{\mathrm{rad}}(\Omega)\). Moreover, we can apply \eqref{radial} in the space $E_{\mathrm{rad}}(\Omega)$ by extending the functions to $\mathbb{R}^N$.
\end{remark}

The following result is essential for establishing the local
H\"older continuity of the gradient of weak solutions.

\begin{theorem}\label{regularidade-C1}
Let $\Gamma\subset\mathbb{R}^N$ be an open set and consider the operator
\[
Qu:=\mathrm{div}A(x,u,\nabla u)+B(x,u,\nabla u).
\]
Define 
\[
W^{1,G}(\Gamma)
:=
\left\{
u\in W_{\mathrm{loc}}^{1,1}(\Gamma):
\int_\Gamma G(|\nabla u|)\dx<\infty
\right\},
\]
where \(
G(t):=\int_0^t g(s)\,\mathrm{d}s\) and \(g\) denote a \(C^1\) function such that
\begin{equation}\label{condicao-g}
    0\leq \frac{t g'(t)}{g(t)}\leq g_0, \quad \forall \, t>0,
\end{equation}
for some constant \(g_0\geq0\). Assume that there exist positive constants \(\Lambda,\Lambda_1,M_0\) and \(\gamma\le 1\) such that
\begin{equation}\label{condicao-aij}
    a^{ij}(x,z,\eta)\xi_i\xi_j
\geq
\frac{g(|\eta|)}{|\eta|}|\xi|^2,
\end{equation}
\begin{equation}\label{condicao-|aij|}
|a^{ij}(x,z,\eta)|
\leq
\Lambda\frac{g(|\eta|)}{|\eta|},    
\end{equation}
\begin{equation}\label{A(x,z,eta)}
    |A(x,z,\eta)-A(y,w,\eta)|
\leq
\Lambda_1\bigl(1+g(|\eta|)\bigr)
\left(
|x-y|^\gamma+|z-w|^\gamma
\right),
\end{equation}
and
\begin{equation}\label{condicao-B}
|B(x,z,\eta)|
\leq
\Lambda_1\bigl(1+g(|\eta|)|\eta|\bigr),    
\end{equation}
for every $x,y\in\Gamma$, $z,w\in[-M_0,M_0]$,
$\eta\in\mathbb{R}^N\setminus\{0\}$, and $\xi\in\mathbb{R}^N$. If $u\in W^{1,G}(\Gamma)\cap L^\infty(\Gamma)$ is a weak solution of
\[
Qu=0, \quad
\text{in } \Gamma,
\]
then \(
u\in C^{1,\beta}(\Gamma)\), for some $\beta>0$.
\end{theorem}

\begin{proof}
    \cite[Theorem 1.7]{Lieberman1991}.
\end{proof}

\begin{lemma}\label{alfa>0}
   Assume that \(\alpha\ge 0\). If \(u\in E_\mathrm{rad}(\mathbb{R}^N)\) is a nonnegative weak solution of \eqref{eq1}, then \(u\in C^{1,\beta}_{\mathrm{loc}}(\mathbb{R}^N)\), for some \(\beta>0\), and \(u>0\) in \(\mathbb{R}^N\).
\end{lemma}

\begin{proof}
    Consider the operator
    \[
    Qu=\mathrm{div}A(\nabla u) +B(x,u),
    \]
    where
    \[
A(\eta)=|\eta|^{p-2}\eta+|\eta|^{q-2}\eta \quad \mbox{and} \quad B(x,z)=|x|^\alpha |z|^{s-2}z.
\]
Moreover, define
\[
g(t):=c_0\left(t^{p-1}+t^{q-1}\right), \quad t>0,
\]
where \(c_0:=\min\{1,p-1\}>0\). See that
\[
g'(t)=c_0(p-1)t^{p-2}+c_0(q-1)t^{q-2}.
\]
Consequently,
\[
0\le \frac{tg'(t)}{g(t)}
=
\frac{(p-1)t^{p-1}+(q-1)t^{q-1}}
{t^{p-1}+t^{q-1}} \leq \frac{(q-1)\left(t^{p-1}+t^{q-1}\right)}
{t^{p-1}+t^{q-1}}=q-1,
\]
because \(p<q\). Thus, the function \(g\) satisfies \eqref{condicao-g}, with \(g_0=q-1>0\).

To prove \eqref{condicao-aij}, notice that
\begin{equation}\label{aij-eta}
    a^{ij}(\eta)
=
\frac{\partial A_i}{\partial \eta_j}(\eta)
=
\left(|\eta|^{p-2}+|\eta|^{q-2}\right)\delta_{ij}
+(p-2)|\eta|^{p-4}\eta_i\eta_j
+(q-2)|\eta|^{q-4}\eta_i\eta_j.
\end{equation}
Hence, for every $\xi\in\mathbb{R}^N$, we can write
\begin{equation}\label{expressao-aij-eta-xi}
    \begin{aligned}
a^{ij}(\eta)\xi_i\xi_j
=
\left(|\eta|^{p-2}+|\eta|^{q-2}\right)|\xi|^2 
+(p-2)|\eta|^{p-4}(\eta\cdot\xi)^2 
+(q-2)|\eta|^{q-4}(\eta\cdot\xi)^2.
\end{aligned}
\end{equation}
If $2\leq p<q$, we have that
\[
a^{ij}(\eta)\xi_i\xi_j
\geq
\left(|\eta|^{p-2}+|\eta|^{q-2}\right)|\xi|^2\geq \frac{g(|\eta|)}{|\eta|}|\xi|^2.
\]

If $1<p<2\leq q$, then using the Cauchy-Schwarz inequality, we have
\[
(\eta\cdot\xi)^2\leq |\eta|^2|\xi|^2.
\]
Since $p-2<0$ and $q-2\geq0$, from \eqref{expressao-aij-eta-xi} it follows that
\[
\begin{aligned}
    a^{ij}(\eta)\xi_i\xi_j &\ge \left(|\eta|^{p-2}+|\eta|^{q-2}\right)|\xi|^2
+(p-2)|\eta|^{p-2}|\xi|^2 = \left((p-1)|\eta|^{p-2}+|\eta|^{q-2}\right)|\xi|^2 
\\
&\ge (p-1)\left(|\eta|^{p-2}+|\eta|^{q-2}\right)|\xi|^2 \ge \frac{g(|\eta|)}{|\eta|}|\xi|^2,
\end{aligned}
\]
where we used the definition \(c_0:=\min\{1,p-1\}>0\). The case $1<p<q<2$ is similar. So, \eqref{condicao-aij} holds.

Now, we show that \eqref{condicao-|aij|} is valid. By \eqref{aij-eta}, we can estimate
\[
\begin{aligned}
|a^{ij}(\eta)|
&\leq
|\eta|^{p-2}+|\eta|^{q-2}
+|p-2||\eta|^{p-2}
+|q-2||\eta|^{q-2}\\
&=
\left(1+|p-2|\right)|\eta|^{p-2}
+\left(1+|q-2|\right)|\eta|^{q-2}\\
&\leq
\max\left\{1+|p-2|,1+|q-2|\right\}
\left(|\eta|^{p-2}+|\eta|^{q-2}\right),
\end{aligned}
\]
where we apply the inequalities \(|\delta_{ij}|\leq 1\) and \(|\eta_i\eta_j|\leq |\eta|^2\). Hence,
\[
|a^{ij}(\eta)| \le \left(\frac{\max\left\{1+|p-2|,1+|q-2|\right\}}{c_0}
\right)\frac{g(|\eta|)}{|\eta|},
\]
and so \eqref{condicao-|aij|} is true.

To prove \eqref{A(x,z,eta)}, notice that, in our case,
\[
A(x,z,\eta)=A(\eta)
=
|\eta|^{p-2}\eta+|\eta|^{q-2}\eta,
\]
and hence $A$ is independent of $x$ and $z$. Therefore, for any
$x,y\in\Gamma$, $z,w\in\mathbb{R}$, and $\eta\in\mathbb{R}^N$,
\[
\begin{aligned}
|A(x,z,\eta)-A(y,w,\eta)|
&=
\left|
\left(|\eta|^{p-2}\eta+|\eta|^{q-2}\eta\right)
-
\left(|\eta|^{p-2}\eta+|\eta|^{q-2}\eta\right)
\right|=0.
\end{aligned}
\]
Since
\[
\Lambda_1\bigl(1+g(|\eta|)\bigr)
\left(
|x-y|^\alpha+|z-w|^\alpha
\right)\geq 0,
\]
for any $\Lambda_1>0$, it follows that
\[
|A(x,z,\eta)-A(y,w,\eta)|
\leq
\Lambda_1\bigl(1+g(|\eta|)\bigr)
\left(
|x-y|^\alpha+|z-w|^\alpha
\right).
\]
Thus, \eqref{A(x,z,eta)} is satisfied.

Finally, we prove \eqref{condicao-B}, with \(\Gamma=B_R\), for each \(R>0\). To do this, since \(\alpha\ge 0\),
\[
|B(x,z)|
\leq
R^\alpha M_0^{s-1}
\bigl(1+g(|\eta|)|\eta|\bigr),
\]
for any \(M_0>0\). Thus, the items \eqref{condicao-g}-\eqref{condicao-B} are valid.

Using the fact that \(u\) is a weak solution of 
\[
    Qu=\mathrm{div}A(\nabla u) +B(x,u), \quad \text{in } \Gamma=B_R,
    \]
because \(C_0^\infty(\Gamma)\subset C_0^\infty(\mathbb{R}^N)\) and \(u \in W^{1,G}(\Gamma)\cap L^\infty(\Gamma)\) (see Lemma \ref{lema-iteracao-de-moser}), we can apply Theorem \ref{regularidade-C1} to get \(
u\in C^{1,\beta}(\Gamma)\), completing the first part of result.

It remains to establish that $u>0$ in $\mathbb{R}^N$. Assume, by
contradiction, that there exists $x_0\in\mathbb{R}^N$ such that
$u(x_0)=0$. Since $u\not\equiv 0$, we may choose
$x_1\in\mathbb{R}^N$ such that $u(x_1)>0$. Let $B_R(x_0)$ be a ball
centered at $x_0$ and containing $x_1$. See that
\[
\operatorname{div}{A}(\nabla u)
=
\Delta_p u+\Delta_q u
=
-|x|^\alpha |u|^{s-2}u
\leq 0,
\quad\text{in } B_R(x_0),
\]
where we have used the fact that $u\geq0$. We next verify the structural conditions needed to apply the Strong
Maximum Principle \cite[Theorem 2.5.1]{PucciSerrin}. Notice first that
\(
{A}\in L^\infty_{\mathrm{loc}}
(\mathbb{R}^N).
\)
We show that there exist constants $a_1,a_2>0$ such that, for every
$|\eta|\leq1$,
\[
\langle{A}(\eta),\eta\rangle
\geq a_1|\eta|^p
\quad\text{and}\quad
|{A}(\eta)|
\leq a_2|\eta|^{p-1}.
\]
Indeed,
\[
\begin{aligned}
\langle{A}(\eta),\eta\rangle=
\left\langle
|\eta|^{p-2}\eta+|\eta|^{q-2}\eta,\eta
\right\rangle=
|\eta|^p+|\eta|^q\geq |\eta|^p.
\end{aligned}
\]
Moreover, since $1<p<q$ and $|\eta|\leq1$, we have
\(
|\eta|^{q-1}\leq|\eta|^{p-1}.
\)
So,
\[
\begin{aligned}
|{A}(\eta)|
=
\left|
|\eta|^{p-2}\eta+|\eta|^{q-2}\eta
\right|\leq
|\eta|^{p-1}+|\eta|^{q-1}\leq
2|\eta|^{p-1}.
\end{aligned}
\]
Thus, the required inequalities hold with \(a_1=1\) and \(a_2=2\). Hence, we can apply \cite[Theorem 2.5.1]{PucciSerrin} to get \(u=0\) in \(B_R(x_0)\) or \(u>0\) in \(B_R(x_0)\), which is a contradiction, and so \(u>0\) in \(\mathbb{R}^N\), completing the proof.
\end{proof}

\begin{lemma}\label{alfa<0}
    Assume that \(-p<\alpha< 0\). If \(u\in E_\mathrm{rad}(\mathbb{R}^N)\) is a nonnegative weak solution of \eqref{eq1}, then it is possible to get a weak solution \(\widetilde{u} \in  E_\mathrm{rad}(\mathbb{R}^N)\) of \eqref{eq1}  such that \(\widetilde{u}>0\) in \(\mathbb{R}^N\) and \(
\left.\widetilde{u}\right|_{\mathbb{R}^{N}\setminus\{0\}} = u.
\)
\end{lemma}

\begin{proof}
By Lemma \ref{lema-iteracao-de-moser} and by adapting
the regularity and positivity arguments of Lemma~\ref{alfa>0}
to annuli centered at the origin, we can conclude that
\begin{equation}\label{alpha-negative-profile-bound}
    u\in C^1(\mathbb R^N\setminus\{0\}) \quad \mbox{and}
\quad
0<u(x) \le M, \quad \forall \, x\ne0,
\end{equation}
for some \(M>0\). Since \(u\) is radial, we can write
\[
u(x)=U(|x|) \in C^1((0,\infty)).
\]
Following the argument used in the proof of \cite[Theorem 1.9]{Cardoso-Carvalho-Alb-Figueiredo} which relies on the radial regularity result of \cite{PucciSerrinZou}, we obtain
\[
\left(|\nabla u|^{p-2}\nabla u+|\nabla u|^{q-2}\nabla u\right)
\in C^1\left(\mathbb{R}^N \setminus \{0\} \right)
\]
and \(u\) solves \eqref{eq1} pointwise in \(\mathbb{R}^N \setminus \{0\}\). Thus, defining 
\begin{equation}
    a(t):=|t|^{p-2}t+|t|^{q-2}t \quad \mbox{and} \quad F(r):=r^{N-1}a(U'(r)),
\end{equation}
for \(r>0\), one can conclude that
\begin{equation}\label{alpha-negative-flux-equation}
F'(r)=-r^{N-1+\alpha}U(r)^{s-1}.
\end{equation}
Fixing
\(r_0>0\), we can write
\[
F(r)
=
F(r_0)
+
\int_r^{r_0}t^{N-1+\alpha}U(t)^{s-1}\,\mathrm{d}t,
\quad 0<r<r_0,
\]
because \(|\int_0^{r_0}t^{N-1+\alpha}U(t)^{s-1}\,\mathrm{d}t|<\infty\). Hence there exists \(\ell \in \mathbb{R}\) such that
\[
\ell:=\lim_{r\to 0^+}F(r).
\]
We claim that \(\ell=0\). Otherwise
\(\ell\ne0\). So, there exist \(C_1>0\) and \(\varepsilon>0\)
such that
\[
|F(r)|\ge C_1,
\quad 0<r<\varepsilon.
\]
Since
\(
|a(t)|=|t|^{p-1}+|t|^{q-1},
\)
we obtain
\[
|U'(r)|^{p-1}+|U'(r)|^{q-1}
=
r^{1-N}|F(r)|
\ge C_1 r^{1-N},
\quad 0<r<\varepsilon.
\]
Using the fact that \(r^{1-N}\to \infty \) as \(r\to0^+\), we can assume that 
\[
|U'(r)|\ge1,
\quad 0<r<\varepsilon.
\]
Since \(p<q\), it follows that
\(
|U'(r)|^{p-1}\le |U'(r)|^{q-1},
\)
and so
\[
2|U'(r)|^{q-1}\ge C_1r^{1-N},
\quad 0<r<\varepsilon.
\]
Thus,
\[
|U'(r)|
\ge C_2 r^{-\frac{N-1}{q-1}},
\quad 0<r<\varepsilon.
\]
Consequently, we can estimate
\[
\begin{aligned}
\infty>\int_0^\varepsilon |U'(r)|^q r^{N-1}\,\mathrm{d}r
&\ge
C_2^q\int_0^\varepsilon
r^{-\frac{q(N-1)}{q-1}+N-1}\,\mathrm{d}r=
C_2^q\int_0^\varepsilon
r^{-\frac{N-1}{q-1}}\,\mathrm{d}r
=\infty,
\end{aligned}
\]
because \(q<N\) implies \(
{(N-1)}/{(q-1)}>1,
\)
which is a contradiction. Therefore,
\begin{equation}\label{alpha-negative-zero-flux}
\lim_{r\to0^+}F(r)=0.
\end{equation}
Integrating \eqref{alpha-negative-flux-equation} on the interval \((\delta,r)\) and taking
\(\delta\to0^+\), we deduce from
\eqref{alpha-negative-zero-flux} that
\begin{equation}\label{alpha-negative-flux-identity}
r^{N-1}
\left(
|U'(r)|^{p-2}U'(r)
+
|U'(r)|^{q-2}U'(r)
\right)
=
-\int_0^r t^{N-1+\alpha}U(t)^{s-1}\,\mathrm{d}t.
\end{equation}
Since \(U(t)>0\) for every \(t>0\), the right-hand side is
strictly negative. Hence, \(U'(r)<0\), for any \(r>0\). So, we can apply \eqref{alpha-negative-profile-bound} to get
\[
0<U(1)\le U(r)\le M,
\quad 0<r<1.
\]
Thus,
\[
0<U(1)\le \lim_{r\to0^+} U(r) \le M.
\]
Now defining
\[
\widetilde u(x):=
\begin{cases}
u(x),&x\ne0,\\[2mm]
L,&x=0,
\end{cases}
\]
where
\(
L:=\lim_{r\to0^+} U(r) \in (0,\infty)
\)
the result is proved.
\end{proof}

\begin{remark}\label{obs:radial-regularity-pieces}
Lemmas \ref{alfa>0} and \ref{alfa<0} remain valid, with similar arguments, when
\(\mathbb{R}^N\) is replaced by a ball, an annulus, or an exterior domain
centered at the origin. Indeed, the \(L^\infty\)-estimate, the verification of
the structural assumptions required for the local \(C^{1,\beta}\)-regularity,
and the strong maximum principle are all local arguments and therefore apply
on these domains as well. In the exterior-domain case, the regularity theorem
is applied on bounded subdomains compactly contained in \(\Omega\), where
\(|x|^\alpha\) is bounded.
\end{remark}

\subsection{Qualitative properties of the critical values}\label{sec4}

This subsection is devoted to the analysis of the behavior of the critical values of $I$ and \(I_\Omega\) as the radii of the domains are varied. These qualitative properties will be used later in the construction of nodal solutions. Our first result concerns the uniform boundedness of the critical levels.

\begin{lemma}\label{uniformbounded}
Let \((\rho_n)_n\subset [0,\infty)\) be such that
$
\rho_n\to \rho,
$
for some \(\rho\in[0,\infty)\). Then the sequence 
$
\left(c(\rho_n,\infty)\right)_n
$
is bounded from above. More precisely, there exist \(K>0\) and \(n_0\in\mathbb{N}\) such that
\[
c_{\mathcal{N}_{\Omega(\rho_n,\infty)}}\leq K,
\quad
\forall \, n\geq n_0.
\]
\end{lemma}

\begin{proof}
Assume that
$
\rho_n\to\rho>0,
$
as $n\rightarrow\infty$. Then there exist \(R>\rho\) and $n_0\in\mathbb{N}$ such that
$
\rho_n<R,
$
for $n\geq n_0$. Let $\varphi\in C_{0,\mathrm{rad}}^\infty(\mathbb{R}^N)
$
be such that
\[
\operatorname{supp}(\varphi)
\subset
\Omega(R,\infty)=
\left\{
x\in\mathbb{R}^N
:
|x|>R
\right\}.
\]
Hence
$
\varphi\in E_{\mathrm{rad}}(\Omega(\rho_n,\infty))\setminus\{0\},
$
for all $n\geq n_0$. In view of Lemma \ref{lemaproj}, for each \(n\ge n_0\) there exists a unique 
$
\tau_n>0
$
such that
$
\tau_n\varphi\in\mathcal{N}_{\Omega(\rho_n,\infty)}.
$
So,
\begin{equation}\label{bounded}
c_{\mathcal{N}_{\Omega}(\rho_n,\infty)}
\leq
I_{\Omega(\rho_n,\infty)}(\tau_n\varphi)=\frac{\tau_n^{p}}{p} \int_{\mathbb{R}^N}|\nabla\varphi|^p\,\mathrm{d}x+\frac{\tau_n^{q}}{q} \int_{\mathbb{R}^N}|\nabla\varphi|^q\,\mathrm{d}x
-
\frac{t_n^s}{s}\int_{\mathbb{R}^N}|x|^\alpha |\varphi|^s\,\mathrm{d}x.
\end{equation}
Since
$
\tau_n\varphi\in\mathcal{N}_{\Omega(\rho_n,\infty)},
$
and $\tau_n>0$ we have that
\[
\tau_n^{p-s} \int_{\mathbb{R}^N}|\nabla\varphi|^p\,\mathrm{d}x+\tau_n^{q-s} \int_{\mathbb{R}^N}|\nabla\varphi|^q\,\mathrm{d}x
=
\int_{\mathbb{R}^N}|x|^\alpha |\varphi|^s\,\mathrm{d}x.
\]
Since
$
p<q<s,
$
the left-hand side tends to \(+\infty\) as \(\tau_n\to0^+\), and tends to \(0\) as \(\tau_n\to+\infty\), yielding a contradiction.  Thus, there exist \(m_0,M_0>0\) such that
\begin{equation}\label{bounded2}
m_0\leq \tau_n\leq M_0, \quad
\forall \, n\geq n_0.
\end{equation}
Therefore, \eqref{bounded} and \eqref{bounded2} imply that there exists \(K>0\) such that
\[
c_{\mathcal{N}_{\Omega(\rho_n,\infty)}}\leq K, \quad
\forall \, n\geq n_0.
\]
The proof is complete.
\end{proof}

\begin{proposition}\label{levels}
The critical levels satisfy the following properties:
\begin{enumerate}
    \item[\((a)\)] if \(0\leq \rho\leq \rho' < \sigma' \leq \sigma \leq \infty\), then
    $
    c_{\mathcal{N}_{\Omega(\rho,\sigma)}}\leq c_{\mathcal{N}_{\Omega(\rho',\sigma')}};
    $
\vspace{0,1cm}
    \item[\((b)\)] if \(\sigma\le\sigma_0<\infty\) and $\sigma-\rho\to 0$, then 
$    c_{\mathcal{N}_{\Omega(\rho,\sigma)}}\to \infty;
    $

\vspace{0,1cm}
    \item[\((c)\)]
    $
    c_{\mathcal{N}_{\Omega(\rho,\infty)}}\to \infty,
    $ as $\rho\to \infty;
    $
\vspace{0,1cm}
    \item[\((d)\)] if $
(\rho_n,\sigma_n)\to(\rho,\sigma)
$, then 
\(
c_{\mathcal{N}_{\Omega(\rho,\sigma)}}
\leq\liminf_{n\to\infty}c_{\mathcal{N}_{\Omega(\rho_n,\sigma_n)}}.
\)
\end{enumerate}
\end{proposition}

\begin{proof}
$(a)$ Assume that
\[
0\leq \rho\leq \rho'<\sigma'\leq \sigma\leq \infty.
\]
Then,
$
\Omega(\rho',\sigma')\subset \Omega(\rho,\sigma).
$
Hence, every function \(u\in E_{\mathrm{rad}}(\Omega(\rho',\sigma'))\) can be identified with its zero extension to \(\Omega(\rho,\sigma)\). So, we have
$
\mathcal{N}_{\Omega(\rho',\sigma')}\subset \mathcal{N}_{\Omega(\rho,\sigma)}.
$
Therefore,
\begin{equation*}
c_{\mathcal{N}_{\Omega(\rho,\sigma)}}
=
\inf_{u\in\mathcal{N}_{\Omega(\rho,\sigma)}} I_{\Omega(\rho,\sigma)}(u)
\leq
\inf_{u\in\mathcal{N}_{\Omega(\rho',\sigma')}} I_{\Omega(\rho',\sigma')}(u)
=
c_{\mathcal{N}_{\Omega(\rho',\sigma')}}   .
\end{equation*}
This proves \((a)\).

\vspace{0,2cm}

\noindent\((b)\) Suppose, by contradiction, that there exist a sequence
$
((\rho_n,\sigma_n))_n 
$
and a constant \(\xi>0\)  such that
$
\sigma_n-\rho_n\to 0,
$ as $n\rightarrow\infty$ and $c_{\mathcal{N}_{\Omega(\rho_n,\sigma_n)}}\leq \xi,$ for $n\in\mathbb{N}$. For each \(n\in\mathbb{N}\), choose
$
u_n\in\mathcal{N}_{\Omega(\rho_n,\sigma_n)}
$
such that
\[
I_{\Omega(\rho_n,\sigma_n)}(u_n)\leq c_{\mathcal{N}_{\Omega(\rho_n,\sigma_n)}}+1\leq \xi+1.
\]
Extending \(u_n\) by zero outside \(\Omega(\rho_n,\sigma_n)\), we may consider \(u_n\) as an element of \(E_{\mathrm{rad}}(\mathbb{R}^N)\). Since \(u_n\in\mathcal{N}_{\Omega(\rho_n,\sigma_n)}\), we have
\[
\begin{aligned}
\xi+1
&\geq
I(u_n)-\frac{1}{s}I'(u_n)u_n=
\left(\frac{1}{p}-\frac{1}{s}\right)
\|\nabla u_n\|_{L^p(\mathbb{R}^N)}^p
+
\left(\frac{1}{q}-\frac{1}{s}\right)
\|\nabla u_n\|_{L^q(\mathbb{R}^N)}^q,
\end{aligned}
\]
which implies that \((u_n)_n\) is bounded in \(E(\mathbb{R}^N)\). 

Now, we claim that
\begin{equation}\label{conv1}
\int_{\mathbb{R}^N}|x|^\alpha |u_n|^s\,\mathrm{d}x\to 0, \quad \mbox{as } n\rightarrow\infty.
\end{equation}
Indeed, since \(u_n\) is supported in \(\Omega(\rho_n,\sigma_n)\), it follows from Corollary \ref{embedding-omega}, H\"older's inequality and the fact that \(p_\alpha^*\leq s<q_\alpha^*\) that
\begin{equation}\label{seq-un-s}
    \begin{aligned}
\int_{\mathbb{R}^N}|x|^\alpha |u_n|^s\,\mathrm{d}x & =\int_{\Omega_n}|x|^\alpha |u_n|^s\,\mathrm{d}x
\\
&\leq
\left(
\int_{\Omega(\rho_n,\sigma_n)}|x|^\alpha |u_n|^{q_\alpha^*}\,\mathrm{d}x
\right)^{s/q_\alpha^*}
\left(
\int_{\Omega(\rho_n,\sigma_n)}|x|^\alpha\,\mathrm{d}x
\right)^{1-s/q_\alpha^*}
\\
&\leq
C_1\|u_n\|_{p,q}^s
\left(
\int_{\Omega(\rho_n,\sigma_n)}|x|^\alpha\,\mathrm{d}x
\right)^{1-s/q_\alpha^*}.
\end{aligned}
\end{equation}
See that
\[
\int_{\Omega(\rho_n,\sigma_n)}|x|^\alpha\,\mathrm{d}x = \omega_N\int_{\rho_n}^{\sigma_n}s^{\alpha +N-1}\,\mathrm{d}s =\frac{\omega_N}{\alpha+N}(\sigma_n^{\alpha+N}-\rho_n^{\alpha+N}).
\]
Since \(\alpha>-p\) and \(p<N\), we have \(\alpha>-N\), and hence
$
r\mapsto r^{N+\alpha-1}
$
belongs to \(L^1(0,\sigma_0)\). Using polar coordinates, we obtain
\[
\int_{\Omega(\rho_n,\sigma_n)}
|x|^\alpha\,\mathrm{d}x
=
\omega_N
\int_{\rho_n}^{\sigma_n}
r^{N+\alpha-1}\,\mathrm{d}r.
\]
Since
$
0\leq \rho_n<\sigma_n\leq\sigma_0
$
and
$
\sigma_n-\rho_n\to0,
$
the absolute continuity of the Lebesgue integral yields
\[
\int_{\rho_n}^{\sigma_n}
r^{N+\alpha-1}\,\mathrm{d}r
\rightarrow0, \quad\text{as }n\to\infty.
\]
Consequently,
\[
\int_{\Omega(\rho_n,\sigma_n)}
|x|^\alpha\,\mathrm{d}x
\to0,
\quad\text{as }n\to\infty.
\]
Combining this convergence with \eqref{seq-un-s}, we conclude that
\eqref{conv1} holds.

On the other hand, since \(u_n\in\mathcal{N}_{\Omega(\rho_n,\sigma_n)}\), we have that
\[
\|\nabla u_n\|_{L^p(\mathbb{R}^N)}^p
+
\|\nabla u_n\|_{L^q(\mathbb{R}^N)}^q
=
\int_{\mathbb{R}^N}|x|^\alpha |u_n|^s\,\mathrm{d}x\rightarrow0, \quad \mbox{as } n\rightarrow\infty,
\]
which implies that 
\begin{equation}\label{conv2}
    \|u_n\|_{p,q}\to 0, \quad \mbox{as } n\rightarrow\infty.
\end{equation} 
However, after extending \(u_n\) by zero outside \(\Omega_n\), we have
$
(u_n)_n\subset\mathcal{N}_{\mathbb{R}^{N}}.
$
By Lemma \ref{bounded-away-zero}, there exists \(\beta>0\) such that
$
\|u_n\|_{p,q}\geq \beta,
$ for all $
n\in\mathbb{N},
$
which contradicts \eqref{conv2}. Therefore,
\[
c_{\mathcal{N}_{\Omega(\rho,\sigma)}}\to\infty,
\quad
\mbox{as } \sigma-\rho\to 0.
\]
This proves \((b)\).

\vspace{0,2cm}

$(c)$ Suppose, by contradiction, that there exist a sequence \((\rho_n)_n\) and a constant \(\xi>0\) such that 
$
\rho_n\to\infty
$
and
$
c_{\mathcal{N}_{\Omega(\rho_n,\infty)}}\leq \xi,
$ for any $n\in\mathbb{N}$. For each \(n\in\mathbb{N}\), choose
$
u_n\in\mathcal{N}_{\Omega(\rho_n,\infty)}
$
such that
\[
I_{\Omega(\rho_n,\infty)}(u_n)\leq c_{\mathcal{N}_{\Omega(\rho_n,\infty)}}+1\leq \xi+1.
\]
Extending \(u_n\) by zero outside \(\Omega(\rho_n,\infty)\), we may consider \(u_n\) as an element of \(E_{\mathrm{rad}}(\mathbb{R}^N)\). Similarly to item $(b)$ one may conclude that \((u_n)_n\) is bounded in \(E_{\mathrm{rad}}(\mathbb{R}^N)\).
We claim that
\begin{equation}\label{conv3}
\int_{\Omega(\rho_n,\infty)}|x|^\alpha |u_n|^s\,\mathrm{d}x\to 0, \quad \mbox{as } n\rightarrow\infty.
\end{equation}
Indeed, by \eqref{radial},
\[
|u_n(x)|
\leq
C_2|x|^{-\frac{N-p}{p}}\|\nabla u_n\|_{L^p(\mathbb{R}^N)},
\qquad x\neq0.
\]
Since \(s>p_\alpha^*\), we obtain
\[
\begin{aligned}
\int_{\Omega(\rho_n,\infty)}
|x|^\alpha |u_n|^s\,\mathrm{d}x
&=
\int_{\Omega(\rho_n,\infty)}
|u_n|^{s-p_\alpha^*}
|x|^\alpha |u_n|^{p_\alpha^*}\,\mathrm{d}x
\\
&\leq
C_2^{\,s-p_\alpha^*}
\|\nabla u_n\|_{L^p(\mathbb{R}^N)}^{s-p_\alpha^*}\int_{\Omega(\rho_n,\infty)}
|x|^{-\frac{(s-p_\alpha^*)(N-p)}{p}}
|x|^\alpha |u_n|^{p_\alpha^*}\,\mathrm{d}x.
\end{aligned}
\]
Set
\[
\delta
:=
\frac{(s-p_\alpha^*)(N-p)}{p}>0.
\]
Since \(|x|\geq\rho_n\) in \(\Omega(\rho_n,\infty)\), we have
$
|x|^{-\delta}\leq \rho_n^{-\delta}.
$
Hence
\[
\int_{\Omega(\rho_n,\infty)}
|x|^\alpha |u_n|^s\,\mathrm{d}x
\leq
C_2^{\,s-p_\alpha^*}
\rho_n^{-\delta}
\|\nabla u_n\|_{L^p(\mathbb{R}^N)}^{s-p_\alpha^*}
\int_{\mathbb{R}^N}
|x|^\alpha |u_n|^{p_\alpha^*}\,\mathrm{d}x.
\]
By the weighted Sobolev embedding at the endpoint \(p_\alpha^*\),
\[
\int_{\mathbb{R}^N}
|x|^\alpha |u_n|^{p_\alpha^*}\,\mathrm{d}x
\leq
C_3
\|\nabla u_n\|_{L^p(\mathbb{R}^N)}^{p_\alpha^*}.
\]
Consequently,
\[
\int_{\Omega(\rho_n,\infty)}
|x|^\alpha |u_n|^s\,\mathrm{d}x
\leq
C_4
\rho_n^{-\delta}
\|\nabla u_n\|_{L^p(\mathbb{R}^N)}^s.
\]
Since \((u_n)_n\) is bounded in
\(E_{\mathrm{rad}}(\mathbb{R}^N)\), there exists \(C_5>0\) such that
\[
\int_{\Omega(\rho_n,\infty)}
|x|^\alpha |u_n|^s\,\mathrm{d}x
\leq
C_5\rho_n^{-\delta}.
\]
Finally, since
$
\delta>0
$
and
$
\rho_n\to\infty,
$
we have
$
\rho_n^{-\delta}\to0,
$
and therefore \eqref{conv3} holds. The remainder of the proof follows as in
item \((b)\).

\vspace{0,2cm}

$(d)$ We have to show that
\begin{equation}\label{goal}
c_{\mathcal{N}_{\Omega(\rho,\sigma)}}
\leq\liminf_{n\to\infty}c_{\mathcal{N}_{\Omega(\rho_n,\sigma_n)}},
\end{equation}
whenever
$
(\rho_n,\sigma_n)\to(\rho,\sigma).
$
For simplicity, we only consider the case
$
\sigma_n=\sigma=\infty,
$
the remaining cases being analogous. 

Assume that
$
\rho_n\to\rho>0,
$
as $n\rightarrow\infty$. For each $n\in\mathbb N$, let
$
u_n\in \mathcal N_{\Omega(\rho_n,\infty)}
$
be such that
\begin{equation}\label{un-nivel-rho_n}
I_{\Omega(\rho_n,\infty)}(u_n)=c_{\mathcal{N}_{\Omega(\rho_n,\infty)}}.    
\end{equation}
Since $u_n\in\mathcal N_{\Omega(\rho_n,\infty)}$ (see \eqref{derivada-Gateaux} with \(u_n=u=v\)), we have that
\begin{equation}\label{nehari-un-igualdade}
    \int_{\Omega(\rho_n,\infty)}
|\nabla u_n|^p\dx+\int_{\Omega(\rho_n,\infty)}
|\nabla u_n|^q\dx=\int_{\Omega(\rho_n,\infty)}
|x|^\alpha |u_n|^s\dx.
\end{equation}
Since $s>q>p$, we may write
\begin{equation}\label{d5}
\begin{aligned}
I_{\Omega(\rho_n,\infty)}(u_n)&=I_{\Omega(\rho_n,\infty)}(u_n)-\frac{1}{s}I_{\Omega(\rho_n,\infty)}'(u_n)u_n
\\
&=
\Bigl(\frac1p-\frac1s\Bigr)\int_{\Omega(\rho_n,\infty)}
|\nabla u_n|^p\dx
+
\Bigl(\frac1q-\frac1s\Bigr)\int_{\Omega(\rho_n,\infty)}
|\nabla u_n|^q\dx
\\
&\geq C_6\left(\int_{\Omega(\rho_n,\infty)}
|\nabla u_n|^p\dx+\int_{\Omega(\rho_n,\infty)}
|\nabla u_n|^q\dx\right).    
\end{aligned}
\end{equation}
Thus, in view of Lemma \ref{uniformbounded} and \eqref{d5},
\[
\int_{\Omega(\rho_n,\infty)}
|\nabla u_n|^p\dx+\int_{\Omega(\rho_n,\infty)}
|\nabla u_n|^q\dx \le C_6^{-1}I_{\Omega(\rho_n,\infty)}(u_n)= C_6^{-1}c_{\mathcal{N}_{\Omega(\rho_n,\infty)}} \leq C_7,
\quad
\forall \, n \in \mathbb{N}.
\]
This together with \eqref{nehari-un-igualdade}, imply that
\begin{equation}\label{d7}
A_n\leq C_7,
\quad
B_n\le C_7,
\quad
C_n\leq 2C_7, \quad \forall \, n\in\mathbb{N},
\end{equation}
where
\[
A_n:=\int_{\Omega(\rho_n,\infty)}
|\nabla u_n|^p\dx, \quad B_n:=\int_{\Omega(\rho_n,\infty)}
|\nabla u_n|^q\dx, \quad \mbox{and} \quad C_n:=\int_{\Omega(\rho_n,\infty)}
|x|^\alpha |u_n|^s\dx.
\]
Now, set
$
\lambda_n:=\rho_n/\rho,
$
and define
\begin{equation}\label{def-vn}
v_n(x):=
u_n(\lambda_n x),
\quad \forall \,
x\in\Omega(\rho,\infty).    
\end{equation}
By Lemma~\ref{lemaproj}, for each \(n\in\mathbb{N}\), there exists a unique number $t_n>0$ such that
$
w_n:=t_n v_n
\in
\mathcal N_{\Omega(\rho,\infty)},
$
which implies that
\begin{equation}\label{I-wn-c}
    I_{\Omega(\rho,\infty)}(w_n)\ge c_{\mathcal{N}_{\Omega(\rho,\infty)}}.
\end{equation}
Therefore, in order to prove \eqref{goal}, it is sufficient to prove that
\begin{equation}\label{goal2}
I_{\Omega(\rho,\infty)}(w_n)-I_{\Omega({\rho_n},\infty)}(u_n)\rightarrow0, \quad \mbox{as } n\rightarrow\infty.
\end{equation}

We claim that $t_n\to1$, as $n\rightarrow\infty$. Since $w_n=t_n v_n\in\mathcal N_{\Omega(\rho,\infty)}$, we have
\begin{equation}\label{d12}
t_n^p
\int_{\Omega(\rho,\infty)}
|\nabla v_n|^p\dx
+
t_n^q
\int_{\Omega(\rho,\infty)}
|\nabla v_n|^q\dx
=
t_n^s
\int_{\Omega(\rho,\infty)}
|x|^\alpha |v_n|^s\dx.
\end{equation}
Using the change of variables $y=\lambda_n x$, it follows from \eqref{def-vn} and \eqref{d12} that
\[
t_n^p\lambda_n^{p-N}A_n
+
t_n^q\lambda_n^{q-N}B_n
=
t_n^s\lambda_n^{-(N+\alpha)}
(A_n+B_n).
\]
Dividing by $A_n+B_n$ yields
\begin{equation}\label{d17}
\theta_n t_n^p\lambda_n^{p-N}
+
(1-\theta_n)t_n^q\lambda_n^{q-N}
=
t_n^s\lambda_n^{-(N+\alpha)},
\end{equation}
where
\[
\theta_n:=
\frac{A_n}{A_n+B_n}\in[0,1].
\]
Passing to a subsequence if necessary, we may assume that
$
\theta_n\to\theta_0\in[0,1].
$
In view of \eqref{d17} and the fact that \(\lambda_n\to1\), as \(n\to\infty\), the sequence $(t_n)_n$ is bounded. Let $t_0$ be any limit point of $(t_n)_n$. Since $\lambda_n\to1$, passing to the limit in \eqref{d17}, we obtain
\[
\theta_0 t_0^p
+
(1-\theta_0)t_0^q
=
t_0^s.
\]
Let $F:(0,\infty)\rightarrow\mathbb{R}$ be defined as
\[
F(t)
=
\theta t^p
+
(1-\theta)t^q
-
t^s.
\]
Since $s>q>p$, it is easy to see that 
$
F(t)>0
$ for $0<t<1$, $F(1)=0$, and $F(t)<0$ for $t>1$. Hence $t=1$ is the unique positive root of $F$, and therefore
$
t_0=1,
$
and the claim is proved.

By writing
\[
I_{\Omega(\rho,\infty)}(v_n)
=
\frac{\lambda_n^{p-N}}{p}A_n
+
\frac{\lambda_n^{q-N}}{q}B_n
-
\frac{\lambda_n^{-(N+\alpha)}}{s}C_n,
\]
it follows from \eqref{d7} that 
\begin{equation*}
|I_{\Omega(\rho,\infty)}(v_n)-I_{\Omega(\rho_n,\infty)}(u_n)| \leq C_8
\Bigl(
|\lambda_n^{p-N}-1|
+
|\lambda_n^{q-N}-1|
+
|\lambda_n^{-(N+\alpha)}-1|
\Bigr),
\end{equation*}
for some constant $C_8>0$ independent of $n$. Since $\lambda_n\to1$, we conclude that
\begin{equation}\label{d20}
I_{\Omega(\rho,\infty)}(v_n)-I_{\Omega(\rho_n,\infty)}(u_n)\rightarrow 0, \quad \mbox{as } n\rightarrow\infty.
\end{equation}
By denoting
\[
\widetilde A_n
=
\int_{\Omega(\rho,\infty)}
|\nabla v_n|^p\dx,
\quad 
\widetilde B_n
=
\int_{\Omega(\rho,\infty)}
|\nabla v_n|^q\dx,
\quad 
\widetilde C_n
=
\int_{\Omega(\rho,\infty)}
|x|^\alpha |v_n|^s\dx,
\]
note that
\[
\widetilde{A}_n=\lambda_n^{p-N}A_n,
\quad
\widetilde{B}_n=\lambda_n^{q-N}B_n,
\quad \mbox{and} \quad 
\widetilde{C}_n=\lambda_n^{-N-\alpha}C_n.
\]
Since $\lambda_n\to1$, it follows from \eqref{d7} that $\widetilde{A}_n$, $\widetilde{B}_n$ and $\widetilde{C}_n$ are uniformly bounded.
Moreover, we may write
\[
I_{\Omega(\rho,\infty)}(w_n)
=
\frac{t_n^p}{p}\widetilde A_n
+
\frac{t_n^q}{q}\widetilde B_n
-
\frac{t_n^s}{s}\widetilde C_n \quad \mbox{and} \quad I_{\Omega(\rho,\infty)}(v_n)
=
\frac1p\widetilde A_n
+
\frac1q\widetilde B_n
-
\frac1s\widetilde C_n.
\]
Hence,
\begin{equation*}
I_{\Omega(\rho,\infty)}(w_n)-I_{\Omega(\rho,\infty)}(v_n)
=
\frac{t_n^p-1}{p}\widetilde A_n
+
\frac{t_n^q-1}{q}\widetilde B_n
-
\frac{t_n^s-1}{s}\widetilde C_n.
\end{equation*}
Using the convergence $t_n\rightarrow1$, we conclude that
\begin{equation}\label{d21}
I_{\Omega(\rho,\infty)}(w_n)-I_{\Omega(\rho,\infty)}(v_n)\rightarrow0, \quad \mbox{as } n\rightarrow\infty.
\end{equation}
Finally, combining \eqref{d20} and \eqref{d21}, we obtain \eqref{goal2} holds, that is,
\[
I_{\Omega(\rho,\infty)}(w_n)-I_{\Omega(\rho_n,\infty)}(u_n)\rightarrow0, \quad \mbox{as } n\rightarrow\infty.
\]
From \eqref{I-wn-c} and \eqref{un-nivel-rho_n},
\[
c_{\mathcal{N}_{\Omega(\rho,\infty)}}
\leq
I_{\Omega(\rho,\infty)}(w_n)
=
I_{\Omega(\rho_n,\infty)}(u_n)+o_n(1)
=
c_{\mathcal{N}_{\Omega(\rho_n,\infty)}}+o_n(1),
\]
which implies that \eqref{goal} holds, with \(\sigma_n=\sigma=\infty\), completing the proof.
\end{proof}

\subsection{Construction of the nodal solutions}

The aim of this subsection is to complete the proof of Theorem \ref{existencianodal}. For an integer $k\geq 1$, we consider
\[
0=\rho_0<\rho_1<\cdots<\rho_k<\rho_{k+1}=\infty.
\]
For each \(i=0,1,\ldots,k\), let
$
c_{\mathcal{N}_{\Omega(\rho_i,\rho_{i+1})}}
$
denote the critical level associated with the functional \(I_{\Omega(\rho_i,\rho_{i+1})}\). Define
\[
\mathcal{C}(\rho_1,\ldots,\rho_k)
:=
\sum_{i=0}^{k}c_{\mathcal{N}_{\Omega(\rho_i,\rho_{i+1})}}
\]
and consider
\[
\Theta:=\inf_{\mathcal{P}_k}\mathcal{C}(\rho_1,\ldots,\rho_k),
\]
where
\[
\mathcal{P}_k:=\left\{
0<\rho_1<\cdots<\rho_k<\infty
\right\}.
\]
Consider a minimizing sequence
\(
\left((\rho_1^n,\ldots,\rho_k^n)\right)_n\subset\mathcal{P}_k
\)
for $\Theta$. Thus,
\begin{equation}\label{seq-minimizante}
    \mathcal{C}(\rho_1^n,\ldots,\rho_k^n)
\to
\Theta,
\quad\text{as } n\to\infty.
\end{equation}
We claim that the sequence \((\rho_k^n)_n\subset \mathbb{R}\) is bounded. Otherwise, \(\rho_k^n \to \infty\), as \(n\to\infty\). By  Proposition \ref{levels} - item \((c)\), we have that
\[  
\mathcal{C}(\rho_1^n,\ldots,\rho_k^n) \ge c_{\mathcal{N}_{\Omega(\rho_k^n,\infty)}} \to \infty,
\]
as \(n\to\infty\), contradicting \eqref{seq-minimizante}. Since \(0<\rho_1^n<\cdots<\rho_k^n\), it follows that for each \(i\in\{1,\ldots,k\}\), the sequence \((\rho_i^n)_n\subset \mathbb{R}\) is bounded. Defining \(\rho_0^n=0\) and \(\rho_{k+1}^n=\infty\), for any \(n\in\mathbb{N}\), we can conclude that up to a subsequence, there exist
\[
0=\bar\rho_0<\bar\rho_1<\cdots<
\bar\rho_k<\bar\rho_{k+1}=\infty
\]
such that \(
\rho_i^n\to\bar\rho_i,
\)
as \(n\to\infty\). Finally, by the lower semicontinuity property (see Proposition \ref{levels} - item \((d)\)), we obtain
\begin{equation}\label{minimizante}
    \begin{aligned}
\Theta & \leq \mathcal{C}(\bar\rho_1,\ldots,\bar\rho_k)=
\sum_{i=0}^{k}
c_{\mathcal{N}_{\Omega(\bar\rho_i,\bar\rho_{i+1})}}\leq
\liminf_{n\to\infty}
\sum_{i=0}^{k}
c_{\mathcal{N}_{\Omega(\rho_i^n,\rho_{i+1}^n)}}
\\
&=
\liminf_{n\to\infty}
\mathcal{C}(\rho_1^n,\ldots,\rho_k^n)=\Theta.
\end{aligned}
\end{equation}
For each \(i=0,1,\ldots,k\), let \(u_i\) be a positive radial solution of \eqref{Sx} on the domain
$
\Omega(\bar\rho_i,\bar\rho_{i+1}).
$
We define \(u:\mathbb{R}^{N}\to\mathbb{R}\) as follows. For
$
x\in\Omega(\bar\rho_i,\bar\rho_{i+1}),
$
$
i=0,1,\ldots,k,
$
we set
\[
u(x)
:=
\begin{cases}
u_i(x), & \text{if } i \text{ is even},\\
-u_i(x), & \text{if } i \text{ is odd}.
\end{cases}
\]
Then \(u\) is radial and changes sign exactly \(k\) times. Moreover, on each domain
$
\Omega(\bar\rho_i,\bar\rho_{i+1}),
$
the function \(u\) satisfies weakly
\[
-\Delta_p u-\Delta_q u
=
|x|^\alpha |u|^{s-2}u.
\]
Equivalently, writing \(u=u(r)\), we have
\begin{equation}\label{radial-equation-nodal}
-\left(r^{N-1}|u'|^{p-2}u'\right)'
-
\left(r^{N-1}|u'|^{q-2}u'\right)'
=
r^{N-1+\alpha}|u|^{s-2}u
\end{equation}
on each interval
$
(\bar\rho_i,\bar\rho_{i+1}),
$
for
$i=0,1,\ldots,k.
$
By Remark \ref{obs:radial-regularity-pieces}, the glued radial function
\(u=u(r)\) belongs to \(C^1((\bar\rho_i,\bar\rho_{i+1}))\) for every
\(i=0,\ldots,k\) satisfies the radial equation pointwise in each nodal
interval. Moreover, since
each one-sign piece has zero trace on the finite boundary components of its
nodal region, \(u\) is continuous at the nodal radii and
\begin{equation}\label{u-se-anulando}
    u(\bar\rho_i)=0,\qquad i=1,\ldots,k.
\end{equation}
Therefore, in order to prove Theorem \ref{existencianodal}, it is sufficient to show that
\begin{equation}\label{limits}
\lim_{r\to\bar\rho_i^-}u'(r)
=
\lim_{r\to\bar \rho_i^+}u'(r),
\quad
i=1,\ldots,k.
\end{equation}
In fact, fix $\varphi\in C_0^\infty(\mathbb R^N)$. Noting that
\[
\begin{aligned}
    \operatorname{div}
\left[
\varphi\left(
|\nabla u|^{p-2}\nabla u
+
|\nabla u|^{q-2}\nabla u
\right)
\right] &= \left(
|\nabla u|^{p-2}\nabla u
+
|\nabla u|^{q-2}\nabla u
\right)\nabla\varphi
\\
&\quad
+
\varphi\operatorname{div}(|\nabla u|^{p-2}\nabla u
+
|\nabla u|^{q-2}),
\end{aligned}
\]
and using the equation in each $
\Omega(\bar\rho_i,\bar\rho_{i+1})$, we obtain
\[
\begin{aligned}
&\int_{\mathbb R^N}
\left[
\left(
|\nabla u|^{p-2}\nabla u
+
|\nabla u|^{q-2}\nabla u
\right)\nabla\varphi
-
|x|^\alpha|u|^{s-2}u\varphi
\right]\dx\\
&\quad=
\sum_{i=0}^{k}
\int_{\Omega(\bar\rho_i,\bar\rho_{i+1})}
\operatorname{div}
\left[
\varphi\left(
|\nabla u|^{p-2}\nabla u
+
|\nabla u|^{q-2}\nabla u
\right)
\right]\dx\\
&\quad=
\sum_{i=0}^{k}
\int_{\partial\Omega(\bar\rho_i,\bar\rho_{i+1})}
\varphi
\left(
|\nabla u|^{p-2}\nabla u
+
|\nabla u|^{q-2}\nabla u
\right)\cdot\nu_i\,\mathrm{d}S,
\end{aligned}
\]
where $\nu_i$ denotes the outward unit normal to $\Omega(\bar\rho_i,\bar\rho_{i+1})$. A simple calculation allows us to conclude that
\[
\begin{aligned}
&\int_{\mathbb R^N}
\left[
\left(
|\nabla u|^{p-2}\nabla u
+
|\nabla u|^{q-2}\nabla u
\right)\nabla\varphi
-
|x|^\alpha|u|^{s-2}u\varphi
\right]\dx\\
&\quad=
\sum_{j=1}^{k}
\Big[
|u'(\rho_j^-)|^{p-2}u'(\rho_j^-)
-
|u'(\rho_j^+)|^{p-2}u'(\rho_j^+)\\
&\hspace{3cm}
+
|u'(\rho_j^-)|^{q-2}u'(\rho_j^-)
-
|u'(\rho_j^+)|^{q-2}u'(\rho_j^+)
\Big]
\int_{|x|=\rho_j}\varphi \,\mathrm{d}S.
\end{aligned}
\]
Hence, if \eqref{limits} holds, then \(u\) is a weak solution of \eqref{eq1}. For simplicity, for $i=1,\ldots,k$, we set
\[
\rho:=\bar\rho_{i-1},
\quad
\tau:=\bar\rho_i,
\quad
\sigma:=\bar\rho_{i+1}.
\]
Without loss of generality, we may assume $u(r)>0$ for $r\in(\rho,\tau)$, and $u(r)<0$ for $r\in(\tau,\sigma)$. Now, define
\begin{equation}\label{derivatives}
u'_-:=
\lim_{r\to\tau^-}u'(r)
\qquad
\mbox{and}
\qquad
u'_+:=
\lim_{r\to\tau^+}u'(r).
\end{equation}
Thus, to prove \eqref{limits}, it is enough to show that
\begin{equation}\label{limits2}
u'_-=u'_+.
\end{equation}
To prove \eqref{limits2}, arguing by contradiction, assume that $u'_-\neq u'_+$. Choose \(\delta>0\) sufficiently small so that
$
\rho<\tau-\delta<\tau+\delta<\sigma.
$
Define
$v:[\rho,\sigma]\to\mathbb{R}$ by
\begin{equation}\label{eq:def-v-linear}
v_\delta(r):=
\begin{cases}
u(r), & \text{if } \rho\leq r\leq \tau-\delta,\\[0.2cm]
u(\tau-\delta)
+
\dfrac{u(\tau+\delta)-u(\tau-\delta)}{2\delta}
\bigl(r-(\tau-\delta)\bigr),
& \text{if } \tau-\delta\leq r\leq \tau+\delta,\\[0.2cm]
u(r), & \text{if } \tau+\delta\leq r\leq \sigma.
\end{cases}
\end{equation}
For simplicity, we denote by \(v(r)\). Since
$
u(\tau-\delta)>0
$ and $u(\tau+\delta)<0$, for \(\delta>0\) small, let $\tau_0=\tau_0(\delta)$ be the unique zero point of $v$ in the interval \((\tau-\delta,\tau+\delta)\). From the definition of \(v\), we have
\begin{equation}\label{isolando-tau-0}
    0
=v(\tau_0)=
u(\tau-\delta)
+
\dfrac{u(\tau+\delta)-u(\tau-\delta)}{2\delta}
\bigl(\tau_0-(\tau-\delta)\bigr).
\end{equation}
Since $u(\tau)=0$, the mean value theorem yields a point
$\xi_\delta^-\in(\tau-\delta,\tau)$ such that
\[
u(\tau-\delta)
=u(\tau)-\delta u'(\xi_\delta^-)
=-\delta u'(\xi_\delta^-).
\]
Since $\xi_\delta^-\to\tau^-$ as $\delta\to0^+$, we have
$u'(\xi_\delta^-)\to u'_-$. So, we can write 
\begin{equation}\label{asymptotics}
u(\tau-\delta)
=
-\delta u'_-+o(\delta)
\quad \mbox{and} \quad 
u(\tau+\delta)
=
\delta u'_+  + o(\delta),
\end{equation}
where \(o(\delta)/\delta\to0\), as \(\delta\to 0^+\). Thus,
\begin{equation}\label{eq:tau0-asymptotic}
\tau_0-\tau
=
\delta\,
\frac{u(\tau-\delta)+u(\tau+\delta)}
{u(\tau-\delta)-u(\tau+\delta)}
=
\frac{u'_- -u'_+}{u'_-+u'_+}\,\delta
+
o(\delta),
\end{equation}
where we used \eqref{isolando-tau-0}. On the interval \((\tau-\delta,\tau+\delta)\), the function \(v\) is linear and satisfies
$
v(\tau_0)=0.
$
Hence,  we obtain
\[
\tau_0-(\tau-\delta)
=
\frac{2\delta u(\tau-\delta)}
{u(\tau-\delta)-u(\tau+\delta)}.
\]
In view of \eqref{asymptotics}, we deduce
\begin{equation}\label{asymptotics3}
    \tau_0-(\tau-\delta)=\frac{|u'_-|}
{\left|\frac{u'_-+u'_+}{2}\right|}
\delta
+
o(\delta).
\end{equation}
Moreover, for every \(r\in(\tau-\delta,\tau+\delta)\), we may write
$
v(r)=v'(r)(r-\tau_0).
$
It follows from \eqref{asymptotics} that
\begin{equation}\label{asymptotics2}
v'(r)
=
\frac{u(\tau+\delta)-u(\tau-\delta)}{2\delta}
=
\frac{u'_-+u'_+}{2}+o(1).
\end{equation}
Therefore, 
\begin{equation}\label{delta-small}
|v'(r)| \le C_0, \quad \mbox{for \(\delta>0\) sufficiently small}.    
\end{equation}
In view of Lemma \ref{lemaproj}, there exist unique constants \(\lambda_\delta>0\) and \(\mu_\delta>0\) such that $\lambda_\delta v\in\mathcal{N}_{\Omega(\rho,\tau_0)}$ and $\mu_\delta v\in\mathcal{N}_{\Omega(\tau_0,\sigma)}$. Define \(\omega:[\rho,\sigma]\to\mathbb{R}\) by
\begin{equation}\label{def-w}
    \omega(r)
:=
\begin{cases}
\lambda_\delta v(r), & \rho\leq r \leq \tau_0,\\[0.2cm]
\mu_\delta v(r), & \tau_0\leq r\leq \sigma.
\end{cases}
\end{equation}
Note that
$
\tau_0=\tau_0(\delta)\to\tau,
$
as $\delta\rightarrow0^{+}$. In order to prove \eqref{limits2}, we shall study the behavior of $\lambda_\delta$, $\mu_\delta$ and some asymptotic estimates.

\begin{lemma}\label{convergencias}
As \(\delta\to0^+\), there hold
\begin{equation}\label{lambda-mu}
    \lambda_\delta\to1
\qquad
\mbox{and}
\qquad
\mu_\delta\to1.
\end{equation}
\end{lemma}

\begin{proof} 
We prove only the first convergence of \eqref{lambda-mu}, since the second one is analogous. Recalling that $\tau_0=\tau_0(\delta)$, set
\[
A_\delta:=
\int_\rho^{\tau_0}|v'|^p r^{N-1}\,\mathrm{d}r,
\quad
B_\delta:=
\int_\rho^{\tau_0}|v'|^q r^{N-1}\,\mathrm{d}r,
\quad
\mbox{and}
\quad
C_\delta:=
\int_\rho^{\tau_0}|v|^s r^{\alpha+N-1}\,\mathrm{d}r.
\]
Similarly, set
\[
A:=
\int_\rho^\tau |u'|^p r^{N-1}\,\mathrm{d}r,
\quad
B:=
\int_\rho^\tau |u'|^q r^{N-1}\,\mathrm{d}r,
\quad
\mbox{and}
\quad
C:=
\int_\rho^\tau |u|^s r^{\alpha+N-1}\,\mathrm{d}r.
\]
Since \(v=u\) on \((\rho,\tau-\delta)\), while the remaining intervals have length \(O(\delta)\) and \(u'\), \(v'\) are bounded near \(\tau\) (see \eqref{delta-small}), we have
\begin{equation}\label{delta-convergences}
A_\delta\to A,
\quad
B_\delta\to B,
\quad
C_\delta\to C, \quad \mbox{as } \delta\rightarrow0^{+}.
\end{equation}
Taking into account that
$
\lambda_\delta v\in\mathcal{N}_{\Omega(\rho,\tau_0)},
$
we have
\begin{equation}\label{delta-identity}
A_\delta\lambda_\delta^{p-s}
+
B_\delta\lambda_\delta^{q-s}
=
C_\delta.
\end{equation}
In view of \eqref{delta-convergences} and \eqref{delta-identity}, we conclude that \((\lambda_\delta)_\delta\) is bounded. Thus, up to a subsequence, we may assume that
$
\lambda_\delta\to\lambda_0>0,
$
as \(\delta\to 0^+\). Since \(u\in\mathcal{N}_{\Omega(\rho,\tau)}\), it follows that
$
A+B=C.
$ 
Hence, passing \eqref{delta-identity} to the limit we conclude that \(\lambda_0=1\), which finishes the proof of \eqref{lambda-mu}.
\end{proof}

\begin{lemma}\label{convergencias}
The following asymptotic estimates hold:
\begin{equation}\label{delta1}
\int_{\tau-\delta}^{\tau}
|u|^s r^{\alpha+N-1}\,\mathrm{d}r=o(\delta)
\quad
\mbox{and}
\quad
\int_{\tau-\delta}^{\tau_0}
\lambda_\delta^s |v|^s r^{\alpha+N-1}\,\mathrm{d}r=o(\delta),
\end{equation}
\begin{equation}\label{delta2}
\int_{\tau-\delta}^{\tau}
\left(|u'|^p+|u'|^q\right)r^{N-1}\,\mathrm{d}r
=
\delta
\left(|u'_-|^p+|u'_-|^q\right)\tau^{N-1}
+
o(\delta),
\end{equation}
\begin{equation}\label{delta3}
\int_{\tau}^{\tau+\delta}
\left(|u'|^p+|u'|^q\right)r^{N-1}\,\mathrm{d}r
=
\delta
\left(|u'_+|^p+|u'_+|^q\right)\tau^{N-1}
+
o(\delta).
\end{equation}
\end{lemma}

\begin{proof}
In order to prove \eqref{delta1}, using that the one-sided derivative
$
u'_-=\lim_{r\to\tau^-}u'(r)
$
exists, there exist \(\delta_0>0\) and \(C_1>0\) such that
$
|u'(r)|\leq C_1,
$
for $r\in(\tau-\delta_0,\tau).
$
Hence, taking into account that $u(\tau)=0$, for \(r\in(\tau-\delta_0,\tau)\), we have
\[
|u(r)|
\leq
\int_r^\tau |u'(t)|\,\mathrm{d}t
\leq
C_1(\tau-r).
\]
Since \(\tau-\delta_0>0\), the function \(r\mapsto r^{\alpha+N-1}\) is bounded on
\((\tau-\delta_0,\tau)\).  Consequently, for \(0<\delta<\delta_0\), we deduce
\[
\begin{aligned}
\int_{\tau-\delta}^{\tau}
|u|^s r^{\alpha+N-1}\,\mathrm{d}r
&\leq
C_1^s\int_{\tau-\delta}^{\tau}
(\tau-r)^s r^{\alpha+N-1}\,\mathrm{d}r\\
&\leq C_2\int_{\tau-\delta}^{\tau}
(\tau-r)^s\,\mathrm{d}r\\
&=
C_3\delta^{s+1}=o(\delta).
\end{aligned}
\]
Using the fact that \(v'\) is bounded for \(\delta>0\) small (see \eqref{delta-small}), the second estimate is obtained in an analogous way. This proves \eqref{delta1}.

To prove \eqref{delta2}, set
$
g(r):=|u'(r)|^p r^{N-1}.
$
It follows from \eqref{derivatives} that
$
g(r)\to |u'_-|^p\tau^{N-1},
$ as $r\to\tau^-.
$
Here we use the existence of the one-sided limit \(u'_-\), together with the continuity of the maps \(t\mapsto |t|^p\) and \(r\mapsto r^{N-1}\). Therefore, we obtain
\[
\begin{aligned}
\int_{\tau-\delta}^{\tau}
|u'|^p r^{N-1}\,\mathrm{d}r
&=
\int_{\tau-\delta}^{\tau}
g(r)\,\mathrm{d}r
=
\int_{\tau-\delta}^{\tau}
\left(g(r)-|u'_-|^p\tau^{N-1}\right)\,\mathrm{d}r
+
\int_{\tau-\delta}^{\tau}
|u'_-|^p\tau^{N-1}\,\mathrm{d}r
\\
&=
\delta |u'_-|^p\tau^{N-1}
+
\int_{\tau-\delta}^{\tau}
\left(g(r)-|u'_-|^p\tau^{N-1}\right)\,\mathrm{d}r\\
& \leq \delta |u'_-|^p\tau^{N-1}
+\delta
\sup_{r\in[\tau-\delta,\tau)}
\left|
g(r)-|u'_-|^p\tau^{N-1}
\right|\\
&= \delta |u'_-|^p\tau^{N-1}
+o(\delta).
\end{aligned}
\]
By setting $
h(r):=|u'(r)|^q r^{N-1}
$
and repeating the argument, we conclude that \eqref{delta2} holds.

The proof of \eqref{delta3} is analogous and therefore we will omit it. The proof is completed
\end{proof}

{For \(0<\rho<\sigma\), we denote by
\[
I_{\rho,\sigma}:=I_{\Omega(\rho,\sigma)}
\]
the energy functional restricted to the domain
\[
\Omega(\rho,\sigma):=\{x\in\mathbb{R}^N:\ 0\le\rho<|x|<\sigma\le \infty\}.
\]
Thus, for a radial function \(u=u(r)\), we have
\[
I_{\rho,\sigma}(u)
=
\omega_N
\int_{\rho}^{\sigma}
\left(
\frac{|u'|^p}{p}
+
\frac{|u'|^q}{q}
-
\frac{|u|^s}{s}r^\alpha
\right)
r^{N-1}\,\mathrm{d}r,
\]
where \(\omega_N=|\mathbb{S}^{N-1}|\).} For the sake of simplicity, henceforth, we ommit $\omega_N$.

\begin{lemma}\label{lem:energy-decrease-omega}
If
$
u'_-\neq u'_+,
$
then, for \(\delta>0\) sufficiently small, there holds
$
I_{\rho,\sigma}(\omega)
<
I_{\rho,\sigma}(u).
$
\end{lemma}

\begin{proof}
We decompose
\begin{equation}\label{decomposicao}
I_{\rho,\sigma}(\omega)-I_{\rho,\sigma}(u)
=
I_1+I_2+I_3,    
\end{equation}
where
\[
I_1
:=
I_{\rho,\tau-\delta}(\omega)
-
I_{\rho,\tau}(u), \quad I_2
:=
I_{\tau+\delta,\sigma}(\omega)
-
I_{\tau,\sigma}(u), \quad I_3
:=
I_{\tau-\delta,\tau+\delta}(\omega).
\]
We first estimate \(I_1\). Taking into account that \(\omega=\lambda_\delta u\) on \((\rho,\tau-\delta)\), we can write
\[
I_1
=
\left[
I_{\rho,\tau}(\lambda_\delta u)
-
I_{\rho,\tau}(u)
\right]
-
I_{\tau-\delta,\tau}(\lambda_\delta u).
\]
Since \(u\in\mathcal{N}_{\Omega(\rho,\tau)}\), by Lemma \ref{lemaproj}, \(I_{\rho,\tau}(u) = \max_{t\ge 0}I_{\rho,\tau}(tu)  \). So,
\[
I_{\rho,\tau}(\lambda_\delta u)
-
I_{\rho,\tau}(u)
\leq 0,
\]
which implies that
\begin{equation}\label{I1}
I_1
\leq
-I_{\tau-\delta,\tau}(\lambda_\delta u).
\end{equation}
Therefore, it follows from \eqref{delta1}, \eqref{delta2}, \eqref{lambda-mu} and \eqref{I1} that
\begin{equation}\label{eq:I1-estimate}
I_1
\leq
-\delta
\left(
\frac{|u'_-|^p}{p}
+
\frac{|u'_-|^q}{q}
\right)\tau^{N-1}
+
o(\delta).
\end{equation}
Analogously, using \eqref{delta3}, one may deduce
\begin{equation}\label{eq:I2-estimate}
I_2
\leq
-\delta
\left(
\frac{|u'_+|^p}{p}
+
\frac{|u'_+|^q}{q}
\right)\tau^{N-1}
+
o(\delta).
\end{equation}

It remains to estimate \(I_3\). Since
$
\tau_0\in(\tau-\delta,\tau+\delta),
$
we have
$
|r-\tau_0|\leq 2\delta.
$
Thus, by the mean value theorem and \eqref{delta-small}, there exists \(C>0\) such that 
\[
|v(r)| = |v(r)-v(\tau_0)| \le C\delta,
\quad
\forall \, r\in(\tau-\delta,\tau+\delta).
\]
Moreover, for \(\tau>0\), we have that \(r^{\alpha+N-1}\) is bounded on \((\tau-\delta,\tau+\delta)\). Therefore,
\[
\begin{aligned}
\int_{\tau-\delta}^{\tau+\delta} \frac{|\omega|^{s}}{s}r^{\alpha+N -1}
&=
\int_{\tau-\delta}^{\tau_0}
\frac{\lambda_\delta^s |v|^s}{s}r^{\alpha+N-1}\,\mathrm{d}r
+
\int_{\tau_0}^{\tau+\delta}
\frac{\mu_\delta^s |v|^s}{s}r^{\alpha+N-1}\,\mathrm{d}r
\\
&\leq
C_1\int_{\tau-\delta}^{\tau+\delta}|v|^s\,\mathrm{d}r
\leq
C_1\int_{\tau-\delta}^{\tau+\delta}\delta^s\,\mathrm{d}r
=
C_2\delta^{s+1}
=
o(\delta).
\end{aligned}
\]

For the sake of convenience, we set
\[
A_\delta:=
\frac{u(\tau+\delta)-u(\tau-\delta)}{2\delta}=\frac{u'_-+u'_+}{2}+o(\delta)=v'(r),
\]
where we have used \eqref{asymptotics2}. For \(m\in\{p,q\}\), by definition of \(\omega\) (see \eqref{def-w}) we have
\begin{equation}\label{est-tau+delta}
   \begin{aligned}
\int_{\tau-\delta}^{\tau+\delta}
|\omega'|^m r^{N-1}\,\mathrm{d}r
&=
\int_{\tau-\delta}^{\tau_0}
\lambda_\delta^m |A_\delta|^m r^{N-1}\,\mathrm{d}r
+
\int_{\tau_0}^{\tau+\delta}
\mu_\delta^m |A_\delta|^m r^{N-1}\,\mathrm{d}r
\\
&=
|A_\delta|^m
\left[
\lambda_\delta^m
\int_{\tau-\delta}^{\tau_0} r^{N-1}\,\mathrm{d}r
+
\mu_\delta^m
\int_{\tau_0}^{\tau+\delta} r^{N-1}\,\mathrm{d}r
\right].
\end{aligned} 
\end{equation}
Since
$
r^{N-1}=\tau^{N-1}+O(\delta),
$
as $\delta\rightarrow0^{+}$, uniformly for \(r\in(\tau-\delta,\tau+\delta)\), we obtain
\[
\int_{\tau-\delta}^{\tau_0} r^{N-1}\,\mathrm{d}r
=
\tau^{N-1}\bigl(\tau_0-(\tau-\delta)\bigr)+o(\delta),
\]
and
\[
\int_{\tau_0}^{\tau+\delta} r^{N-1}\,\mathrm{d}r
=
\tau^{N-1}\bigl(\tau+\delta-\tau_0\bigr)+o(\delta).
\]
Thus, using \eqref{lambda-mu} in \eqref{est-tau+delta} we deduce
\[
\begin{aligned}
\int_{\tau-\delta}^{\tau+\delta}
|\omega'|^m r^{N-1}\,\mathrm{d}r
&=
|A_\delta|^m
\left[
\int_{\tau-\delta}^{\tau_0}r^{N-1}\,\mathrm{d}r
+
\int_{\tau_0}^{\tau+\delta}r^{N-1}\,\mathrm{d}r
+
o(\delta)
\right]
\\
&=
\left(
\left|\frac{u'_-+u'_+}{2}\right|^m+o_\delta(1)
\right)
\left[
2\delta\,\tau^{N-1}
+
o(\delta)
\right]
\\
&=
2\delta\,\tau^{N-1}
\left|\frac{u'_-+u'_+}{2}\right|^m
+
o(\delta).
\end{aligned}
\]
Therefore, we conclude that
\begin{equation}\label{eq:I3-estimate}
I_3
=
I_{\tau-\delta,\tau+\delta}(\omega)
=
2\delta\,\tau^{N-1}
\left(
\frac{1}{p}
\left|\frac{u'_-+u'_+}{2}\right|^p
+
\frac{1}{q}
\left|\frac{u'_-+u'_+}{2}\right|^q
\right)
+
o(\delta).
\end{equation}
Combining \eqref{decomposicao} together with  \eqref{eq:I1-estimate}-\eqref{eq:I3-estimate}, we obtain
\[
\begin{aligned}
I_{\rho,\sigma}(\omega)-I_{\rho,\sigma}(u)
&\leq
-\delta\tau^{N-1}
\left(
\frac{|u'_-|^p}{p}
+
\frac{|u'_-|^q}{q}
+
\frac{|u'_+|^p}{p}
+
\frac{|u'_+|^q}{q}
\right)
\\
&\quad
+
2\delta\,\tau^{N-1}
\left(
\frac{1}{p}
\left|\frac{u'_-+u'_+}{2}\right|^p
+
\frac{1}{q}
\left|\frac{u'_-+u'_+}{2}\right|^q
\right)
+
o(\delta).
\end{aligned}
\]
Equivalently, 
\[
I_{\rho,\sigma}(\omega)-I_{\rho,\sigma}(u)
\leq
-\xi_0\delta+o(\delta),
\]
where
\[
\begin{aligned}
\xi_0
:=
2\tau^{N-1}
\bigg[
&\frac{1}{p}
\left(
\frac{|u'_-|^p+|u'_+|^p}{2}
-
\left|\frac{u'_-+u'_+}{2}\right|^p
\right)
+
\frac{1}{q}
\left(
\frac{|u'_-|^q+|u'_+|^q}{2}
-
\left|\frac{u'_-+u'_+}{2}\right|^q
\right)
\bigg].
\end{aligned}
\]
Recall that $u(\tau)=0$, $u(r)>0$ in $(\rho,\tau)$, and $u(r)<0$ in $(\tau,\sigma)$. Hence, $u'_-<0$ and $u'_+<0$, and, by assumption,
$
u'_-\neq u'_+,
$
we have
$
|u'_-|\neq |u'_+|.
$
Therefore, by the strict convexity of the functions
$
t\mapsto t^p
$ and $t\mapsto t^q$
on \([0,\infty)\), we conclude that
$
\xi_0>0.
$
Consequently,
\[
I_{\rho,\sigma}(\omega)-I_{\rho,\sigma}(u)<0,
\]
for \(\delta>0\) sufficiently small. Therefore
\[
I_{\rho,\sigma}(\omega)<I_{\rho,\sigma}(u),
\]
and the proof is complete. 
\end{proof}

\subsection{Proof of Theorem \ref{existencianodal}} In the construction of $u$, choose each positive radial solution
$u_j$ so that
\[
I_{\bar\rho_j,\bar\rho_{j+1}}(u_j)
=
c_{\mathcal N_{\Omega(\bar\rho_j,\bar\rho_{j+1})}},
\qquad j=0,\ldots,k.
\]
See that
\[
I_{\rho_0,\infty}(u)
=
\sum_{j=0}^{k}
c_{\mathcal N_{\Omega(\bar\rho_j,\bar\rho_{j+1})}}
=
\mathcal{C}(\bar\rho_1,\ldots,\bar\rho_k),
\quad \rho_0=\bar\rho_0=0.
\]
Extend $\omega$ by $u$ outside $(\rho,\sigma)$. By construction,
\[
\omega|_{\Omega(\rho,\tau_0)}
\in\mathcal N_{\Omega(\rho,\tau_0)},
\qquad
\omega|_{\Omega(\tau_0,\sigma)}
\in\mathcal N_{\Omega(\tau_0,\sigma)}.
\]
Replacing $\bar\rho_i=\tau$ by $\tau_0$ yields a partition in \(\mathcal{P}_k\). From \eqref{minimizante},
\begin{equation}\label{final}
    \begin{aligned}
I_{\rho_0,\infty}(u)
&=\mathcal{C}(\bar\rho_1,\ldots,\bar\rho_k)=\Theta:=\inf_{\mathcal{P}_k}\mathcal{C}(\rho_1,\ldots,\rho_k)\le \mathcal{C}(\bar\rho_1,\ldots,\bar\rho_{i-1},
\tau_0,\bar\rho_{i+1},\ldots,\bar\rho_k)\\
&\le
\sum_{\substack{0\le j\le k\\j\notin\{i-1,i\}}}
I_{\bar\rho_j,\bar\rho_{j+1}}(u)
+I_{\rho,\tau_0}(\omega)+I_{\tau_0,\sigma}(\omega)
=I_{\rho_0,\infty}(\omega).
\end{aligned}
\end{equation}
Assume, by contradiction, that
\[
u'_-\neq u'_+.
\]
Then, by Lemma \ref{lem:energy-decrease-omega}, for \(\delta>0\) sufficiently small, we have
\[
I_{\rho_0,\infty}(\omega)<I_{\rho_0,\infty}(u),
\]
which contradicts \eqref{final}. Therefore,
\[
u'_-=u'_+.
\]
Finally, by construction, \(u\) is a weak solution of Problem \eqref{eq1} that changes sign exactly at the radii
$
\rho_1,\ldots,\rho_k.
$
Thus \(u\) has exactly \(k+1\) nodal domains, or equivalently \(k\) nodal zeros in the radial variable. This proves Theorem \ref{existencianodal}.

\vspace{0.5cm}
 
\noindent\textbf{Conflict of interest.}
The authors declare that he has no known competing financial interests or personal relationships that could have appeared to influence the work reported in this work.

\medskip

\noindent\textbf{Data availability.}
This work has no associated data.

\medskip

\noindent\textbf{Funding.}
The first author was partially supported by CNPq with grant 301849/2025-0. The second author was partially supported by CNPq with grant 300853/2025-4. The third author was partially supported by FACEPE/Brazil Proc. APQ-0967-1.01/24.

\end{document}